\let\ORIlabel\label
\let\ORIrefstepcounter\refstepcounter
\AddToHook{package/hyperref/before}{
   \let\label\ORIlabel 
   \let\refstepcounter\ORIrefstepcounter}

\documentclass[onefignum,onetabnum]{siamart220329}

\usepackage{array} 
\usepackage{longtable} 
\usepackage{multirow} 
\usepackage{etoolbox} 
\usepackage{lipsum,verbatim,tcolorbox}
\usepackage{wrapfig}
\usepackage{amsfonts}
\usepackage{graphicx}
\usepackage{epstopdf}
\usepackage{algorithmic}
\usepackage{subcaption}
\usepackage{colortbl, hhline}
\usepackage{booktabs} 

\usepackage{ntheorem,cleveref}

\renewtheorem{lemma}[theorem]{Lemma}
\renewtheorem{corollary}[theorem]{Corollary}
\renewtheorem{proposition}[theorem]{Proposition}
\renewtheorem{definition}[theorem]{Definition}

\newcommand{\bxi}{\boldsymbol{\xi}}
\ifpdf
  \DeclareGraphicsExtensions{.eps,.pdf,.png,.jpg}
\else
  \DeclareGraphicsExtensions{.eps}
\fi
\usepackage[textsize=footnotesize,textwidth=0.8in]{todonotes}

\newcommand{\tabincell}[2]{\begin{tabular}[t]{#1}#2\end{tabular}}

 \usepackage{pifont}
 
 \usepackage{mathabx}
 \newcommand{\cmark}{\ding{51}}%
\newcommand{\xmark}{\ding{55}}%

\newcommand{\bv}{\boldsymbol v}
\newcommand{\bz}{\boldsymbol z}

\newcommand{\Xscr}{\mathcal{X}}
\newcommand{\Cscr}{\mathcal{C}}

\usepackage{amsmath,multirow}
\usepackage{algorithm}
\usepackage{algorithmic}
\usepackage{appendix}

\newcommand{\lm}[1]{\begin{color}{blue}#1\end{color}}

\newsiamremark{remark}{Remark}
\newsiamremark{hypothesis}{Hypothesis}
\crefname{hypothesis}{Hypothesis}{Hypotheses}
\newsiamthm{claim}{Claim}
\newsiamthm{assumption}{Assumption}

\headers{Zeroth-order methods}{ }

\title{Zeroth-order Gradient and Quasi-Newton Methods for Nonsmooth Nonconvex Stochastic Optimization\thanks{Submitted \today. A very preliminary version appeared
in~\cite{shanbhag2021zeroth}{.} 
\funding{This work was funded in part by the following grants:  ONR grants N$00014$-$22$-$1$-$2589$ and N$00014$-$22$-$1$-$2757$, AFOSR grant FA9550-24-1-0259, and  DOE grant DE-SC$0023303$.}}
}

\author{LUKE MARRINAN\and UDAY V. SHANBHAG\thanks{LM is with Ind. \& Manuf. Engg., Pennsylvania State University  ({\email{lwm5431@psu.edu}}), UVS is with Ind. \& Oper. Engg., University of Michigan (\email{udaybag@umich.edu}), while FY is with Ind. \& Sys. Engg., Rutgers University (\email{farzad.yousefian@rutgers.edu}). }
\and FARZAD YOUSEFIAN
}

\usepackage{amsopn}

\makeatletter
\newcommand*{\addFileDependency}[1]{
  \typeout{(#1)}
  \@addtofilelist{#1}
  \IfFileExists{#1}{}{\typeout{No file #1.}}
}
\makeatother

\usepackage{thmtools}

\declaretheoremstyle[
  headfont=\normalfont\bfseries,
  notefont=\normalfont\bfseries, 
  headindent = 0pt,
  notebraces={[}{]},
  bodyfont=\normalfont,
numberwithin=section,
spaceabove=8pt,
mdframed={
 linecolor=gray!20,
  skipabove=18pt,
  innerbottommargin=10pt, 
  backgroundcolor=gray!20,
  innerleftmargin=3pt,
  innerrightmargin=8pt}
]{mytheorem}

\ifpdf
\hypersetup{
  pdftitle={},
  pdfauthor={}
}
\fi

\def\bko{{\rm 1\kern-.17em l}}

\def\Nscr{{\mathcal N}}

\def\Kscr{{\mathcal K}}

\newcommand{\f}{\tilde{f}}

\def\be{\begin{enumerate}}
\def\ee{\end{enumerate}}
\def\Nscr{{\mathcal N}}

 \newcommand{\remove}[1]{}
\newcommand{\EXP}[1]{\mathbb{E}\left[#1\right] }

\newcommand{\x}{{\mathbf{x}}}
\newcommand{\uu}{{\mathbf{u}}}
\newcommand{\vv}{{\mathbf{v}}}
\newcommand{\y}{{\mathbf{y}}}

\def\sF{\mathcal{F}}

\def\Real{\mathbb{R}}
\def\g{\gamma}

\definecolor{britishracinggreen}{rgb}{0.0, 0.36, 0.15}
\newcommand{\fy}[1]{{\color{black}#1}}
\newcommand{\fyy}[1]{\begin{color}{black}#1\end{color}}
\newcommand{\far}[1]{\begin{color}{black}#1\end{color}}
\newcommand{\us}[1]{\begin{color}{black}#1\end{color}}

 \usepackage{graphicx}

\usepackage{tikz}
\usetikzlibrary{decorations.pathreplacing,calc}
\newcommand{\tikzmark}[1]{\tikz[overlay,remember picture] \node (#1) {};}

\newcommand{\vrgzoIterationComplexity}{$\mathcal{O}({n^{1/2}}{(L_0\eta^{-1} +L_0^2)} \epsilon^{-2})$}
\newcommand{\vrgzoSampleComplexity}{$\mathcal{O}({n^{3/2}(L_0^3\eta^{-2}+L_0^5)} \epsilon^{-4})$}
\newcommand{\vrsqnzoIterationComplexity}{$\mathcal{O}({L_0^{4}}{n^{2}}{\eta^{-4}}\epsilon^{-2})$}
\newcommand{\vrsqnzoSampleComplexity}{$\mathcal{O}(L_0^{9} n^{5}\eta^{-5}\epsilon^{-5})$}
\begin{document}

\maketitle
\thispagestyle{empty}

\begin{abstract}  
	We consider the minimization of a Lipschitz continuous and
	expectation-valued function, denoted by $f$ and defined as $f(\x) \,
	\triangleq \, \mathbb{E}[\, \tilde{f}(\x,  \bxi)\, ]$, over a closed and convex
	set {$\Xscr$}. Our focus lies on obtaining both asymptotics as well as rate and
	complexity guarantees for computing an approximate stationary point (in a
	Clarke sense) via zeroth-order schemes. We adopt a smoothing-based approach
	reliant on minimizing $f_{\eta}$ where $f_{\eta}(\x) \triangleq
	\mathbb{E}_{\uu}\left[\, f(\x+\eta \uu)\, \right]$, $\uu$ is a random variable defined on
	a unit sphere, and $\eta > 0$.  In fact, it has been observed that a
	stationary point of the $\eta$-smoothed problem is an $\eta$-stationary
	point for the original problem in the Clarke sense. In such a setting, we
	develop two sets of schemes with promising empirical behavior. (I) {We
		develop a smoothing-enabled {\em variance-reduced zeroth-order gradient}
		framework ({\bf VRG-ZO}) for minimizing $f_{\eta}$ over $\Xscr$}. In this
	setting, we make two sets of contributions for the sequence generated by
	the proposed zeroth-order gradient scheme.  (a) The residual function of
	the smoothed problem tends to zero almost surely along the generated
	sequence, allowing for making guarantees {for $\eta$-Clarke stationary
		solutions of the original problem}; (b) To compute an $\x$ that ensures
	that the expected norm of the residual of the $\eta$-smoothed problem is
	within $\epsilon$ requires no greater than \vrgzoIterationComplexity ~projection steps and \vrgzoSampleComplexity ~function evaluations.  (II) Our second scheme is
		a zeroth-order stochastic quasi-Newton scheme ({\bf VRSQN-ZO}) reliant on a
	combination of randomized and Moreau smoothing; the corresponding iteration and sample complexities for this scheme are \vrsqnzoIterationComplexity ~and  
	\vrsqnzoSampleComplexity, respectively. These
	statements appear to be novel in the case of both constrained problems as well as in SQN settings, as there appear to be few available results that can contend
	with general nonsmooth, nonconvex, and stochastic regimes via zeroth-order
	approaches. 
\end{abstract}
\section{Introduction}
We consider the following stochastic optimization problem
\begin{align}\label{eqn:prob}
	\min_{\x \, \in \, \Xscr} & \quad f(\x) \, \triangleq \, \mathbb{E}\left[\, \f(\x,\bxi)\, \right], 
\end{align}
where $f:\mathbb{R}^n\to \mathbb{R}$ is a real-valued, nonsmooth, and nonconvex
function, $\Xscr \subseteq \mathbb{R}^{n} $ is a closed and convex set, and $\bxi:\Omega \to \mathbb{R}^d$ denotes
a random variable associated with the probability space $(\Omega,
\mathcal{F},\mathbb{P})$. {Our interest lies in the case when $f$ is
	Lipschitz continuous and at any $\x \in \mathbb{R}^n$,  $f(\x)\triangleq
	\mathbb{E}[{\f}(\x,\bxi)]$ and $\bxi: \Omega \mapsto \Xi$ is a $d$-dimensional
	random variable with $\Xi \, \subseteq \, \Real^d$. We denote a realization of
	$\bxi$ and $\f({\x},\bxi)$ by $\xi$ and $\f({\x},\xi)$, respectively}. Recall that
function $f$ is $L_0$-Lipschitz continuous on the set $\Xscr$ if there exists a
scalar $L_0>0$ such that for all $\x,\y \in \Xscr$ we have $|f(\x)-f(\y)| \leq
L_0\|\x-\y\|$. Our goal lies in devising randomized zeroth-order
stochastic gradient and stochastic quasi-Newton (SQN) methods with iteration
and sample complexity guarantees for approximating a stationary point to
\eqref{eqn:prob} in nonsmooth nonconvex regimes. Despite significant
literature on contending with nonsmooth stochastic convex optimization
problems~\cite{shapiro09lectures}, most nonconvex generalizations are
restricted to structured regimes where nonconvexity often emerges as  an
expectation-valued smooth function while nonsmoothness arises in a
deterministic form.  Often, $\tilde{f}(\bullet,\xi)$ may be both nonconvex
and nonsmooth and proximal stochastic gradient
schemes~\cite{ghadimi2016mini,lei2020asynchronous} cannot be adopted.
We now discuss some relevant research in nonsmooth and nonconvex regimes. 

\begin{table}[htp]
	\caption{Previous work on stochastic, nonsmooth, nonconvex settings. 
	} \label{table:nonsmoothnonconvex}
	\begin{center}
		\vspace{-0.1in}
			\tiny
\begin{tabular}{>{\raggedright\arraybackslash}p{0.8in}cccccp{1.6in}}				
				\toprule
				Author  & C & S & D  & U & O  & Comments \\
				\hline
				\tabincell{c}{\\Burke, Lewis, \\
					Overton (2005)}   & \cmark &  \xmark  &  \cmark &   \cmark &  \cmark & Gradients in neighborhood of iterate are assumed to be available. \\ \midrule
					Xu and Yin (2017)  &   \xmark & \tabincell{c}{ \xmark \\nonconvex\\ part is \\ proximable} &  \cmark &  \cmark &  \cmark & Prox-linear algorithm applied to nonconvex problem.   \\ \midrule
				Ghadimi, Lan, Zhang (2016)  & \tabincell{c}{ \xmark \\ nonconvex\\ part is\\ smooth} &  \tabincell{c}{\xmark \\ nonsmooth \\ part is \\ convex} &  \cmark  & \cmark  &  \cmark & A randomized zeroth-order algorithm leveraging Gaussian smoothing is introduced.  \\ \midrule
				Nesterov and Spokoiny (2017)  &  \cmark  &  \xmark  &  \cmark   &   \cmark  &  \xmark & Zeroth-order  Gaussian smoothing \\ \midrule		
				Zhang, Lin, Jegelka, Sra, Jadbabaie (2020)   &  \xmark  &  \xmark &  \cmark &   \cmark  &  \cmark   & Negative result for nonconvex nonsmooth optimization.  \\ \midrule
				\rowcolor{pink}
				\tabincell{c}{\\
					{\texttt{\bf VRG-ZO}}} &  \xmark  &  \xmark &  \xmark &   \xmark  &  \xmark   & Zeroth-order method reliant on randomized smoothing. \\
				\bottomrule
			\end{tabular}
	\end{center}
	{\tiny C:{C}onvex, S:{S}mooth, D:Deterministic, U: Unconstrained, O: Access to Oracle for subdifferential.} 
	\vspace{-0.2in}
\end{table}

\begin{table}[htp]
	\caption{Previous work on SQN in convex, nonconvex, smooth, nonsmooth settings.} \label{table:SQN}
	\begin{center}
		\tiny
		\vspace{-0.1in}
			\begin{tabular}{>{\raggedright\arraybackslash}p{0.8in}cccccp{2.25in}} 
				\toprule
				Author  & C & S & D  & U & O   & Comments \\
				\midrule
				Wang, Ma, Goldfarb, Liu (2017)  & \xmark  & \cmark & \cmark & \cmark &  \cmark & A stochastic, limited memory BFGS scheme for nonconvex optimization is introduced. \\ \midrule				
				Curtis and Que (2019)  & \xmark  & \xmark & \cmark & \cmark & \cmark & Convergence for an SQN schemes based on gradient sampling is delivered.\\ \midrule
				Bollapragada and Wild (2023) &  \cmark & \cmark & \xmark & \cmark  & \xmark & An adaptive sampling zeroth-order quasi-Newton method is proposed.  \\ \midrule \rowcolor{pink}	    \tabincell{c}{\\
					{\texttt{\bf VRSQN-ZO}}}	 & \xmark  & \xmark & \xmark &  \xmark  & \xmark   & Zeroth-order SQN reliant on randomized and Moreau smoothing.\\
				\bottomrule
			\end{tabular}
	\end{center}
	{\tiny C:{C}onvex, S:{S}mooth, D:Deterministic, U: Unconstrained, O: Access to Oracle for subdifferential.} 
	\vspace{-0.2in}
\end{table}
	\begin{table}[htb]
			\caption{Comparison of rate and complexity guarantees for gradient-based schemes.}
		\label{tbl:summary_table}
		\scriptsize
		\centering
		\begin{tabular}{>{\raggedright\arraybackslash}p{2.05cm}>{\centering\arraybackslash}p{3.25cm}>{\centering\arraybackslash}p{2.9cm}>{\raggedright\arraybackslash}p{3cm}}
			\toprule
			\textbf{Authors} & \textbf{Conv. Rate / Oracle Complexity }  & \textbf{Batch Size} / $N_k$ &\textbf{Comments} \\
			\midrule
			Lin, Zheng, Jordan~\cite{lin2022} & {${\mathcal{O}}\left( n^{\tfrac{3}{2}} \left( \frac{L_{0}^4}{\epsilon^4} + \frac{  L_{0}^3}{{\eta} \epsilon^4} \right) \right)$}  & $N_k=B$ & {$\mathbb{E}_{\xi}[L(\bxi)^2] \leq L_{0}^2$} \\
			\hline \addlinespace
			Kornowski and Shamir~\cite{shamir2024} & {${\mathcal{O}}\left(\frac{{n} L_{0}^2 }{{\eta \epsilon^3}}\right)$}   & $N_k = 1$ & {Nonsmooth scheme~\cite{cutkosky2023} to improve dimension dependence by $n^{1/2}$.} \\
\hline \addlinespace
			Chen, Xu, Luo~\cite{shamir2024} & ${\mathcal{O}}\left(n^{\tfrac{3}{2}} \left(\frac{L_{0}^3}{ {\epsilon^{3}}} + \frac{ L_{0}^2 }{{\eta \epsilon^{3}}}\right)\right)$    & $N_k= \mathcal{O}\left(\epsilon^{-2} \right)$  & Employs variance reduction to improve rate. \\
			\addlinespace
\hline \rowcolor{pink}
			\texttt{VRG-ZO} / current paper, \cite{shanbhag2021zeroth} & $\mathcal{O}\left(n^{1/2}\left(\tfrac{L_0}{\eta \epsilon^2}+ \tfrac{L_0^2}{\epsilon^2}\right)\right)$ / $\mathcal{O}\left(n^{3/2}\left(\tfrac{L_0^3}{\eta^2 \epsilon^4}+ \tfrac{L_0^5}{\epsilon^4}\right)\right)$ & $N_k := \lceil a \sqrt{n} L_0 \left(k+1\right)  \rceil$ & 
Addresses the constrained setting.			\\
			\bottomrule
		\end{tabular}
\end{table}
\noindent {\em (a) Nonsmooth nonconvex optimization.}
In~\cite{burke02approximating}, Burke, Lewis, and Overton demonstrate how gradient sampling schemes
allow for approximating the Clarke subdifferential of a
function that is differentiable almost everywhere and proceed to  
{develop a robust {\em gradient sampling} scheme~\cite{burke05robust}, proving    
	subsequential convergence to an $\epsilon$-Clarke stationary point when $f$ is
	locally Lipschitz. Further, in~\cite{kiwiel07convergence}, Kiwiel proves {sequential convergence to
		Clarke stationary points} without requiring compactness of level sets. There have also been efforts
	to develop statements in structured regimes where $f$ is either weakly
	convex~\cite{davis19stochastic,davis19proximally} or $f = g+h$ and $h$ is
	smooth and possibly nonconvex while $g$ is convex, nonsmooth,
	and proximable~\cite{xu2017globally,lei2020asynchronous}.
	
	\smallskip    
	
	\noindent {\em (b) Nonsmooth nonconvex stochastic optimization.} Prior to 2021, most
	efforts have been restricted to structured regimes where $f(\x) = h(\x) +
	g(\x)$, $h(\x) \triangleq \mathbb{E}[\, \tilde{h}(\x,\bxi)\, ]$, $h$ is smooth
	and possibly nonconvex while $g$ is closed, convex, and proper with an
	efficient proximal evaluation. In such settings, proximal stochastic gradient
	techniques~\cite{ghadimi2016mini} and their variance-reduced
	counterparts~\cite{ghadimi2016mini,lei2020asynchronous} were developed.
	{More recently, a comprehensive examination of nonsmooth and nonconvex
		problems can be found in the monograph~\cite{cui2021modern}; related efforts
		for resolving probabilistic and risk-averse settings via sample-average
		approximation approaches have been examined
		in~\cite{cui2022nonconvex,qi2022asymptotic}. {Since 2021, beginning with an earlier version of the first part of this work~\cite{shanbhag2021zeroth}, schemes for the unstructured problem setting have been examined in \cite{shamir2024, chen2024, lin2022}, albeit in the unconstrained setting.  We summarize the results of these three works in {Table~}\ref{tbl:summary_table}}.} 
	
	\smallskip    
	
	\noindent {\em (c) Zeroth-order and smoothing methods.} Deterministic~\cite{chen12smoothing}
	and randomized smoothing~\cite{steklov1,ermoliev95minimization} have been the
	basis for resolving a broad class of nonsmooth optimization
	problems~\cite{mayne84,yousefian2010convex,yousefian2012stochastic}. Such avenues have also emerged in the resolution of stochastic variational and game-theoretic problems~\cite{yousefian2013regularized,yousefian2015self,yousefian2017smoothing}. When the original objective
	function is nonsmooth and nonconvex, Nesterov and Spokoiny~\cite{nesterov17}
	examine unconstrained nonsmooth and nonconvex optimization problems via
	Gaussian smoothing.  In~\cite{shamir2024}, a zeroth-order scheme with iteration and sample complexity ${\mathcal{O}}\left(\tfrac{{n} L_{0}^2 }{{\eta \epsilon^3}}\right)$ has been proposed,  achieving an optimal linear dependence on $n$ in producing $(\eta,\epsilon)$-stationary points for unconstrained nonsmooth nonconvex stochastic optimization.  Zeroth-order stochastic approximation methods for smooth nonconvex
	stochastic optimization in the constrained setting has also been
	proposed~\cite{Balasubramanian2022}, where Gaussian smoothing is employed.  Despite recent progress in both the unconstrained   and the smooth regimes,  zeroth-order schemes with guarantees for resolving constrained nonsmooth nonconvex stochastic optimization have not been adequately studied. 
	
	\medskip
	
	\noindent {\em (d) SQN methods.} 
	{(i) \em Convex settings.} Amongst the earliest efforts to contend with uncertainty were suggested by Mokhtari and Ribiero~\cite{Mokhtari14}, who proposed a regularized stochastic variant of the BFGS algorithm for minimizing strongly convex expectation-valued objective functions. These efforts were expanded upon in~\cite{Byrd,mokhtari2015global} where limited memory variants of the stochastic BFGS scheme were developed. By incorporating iterative regularization, Yousefian et al. derived rate guarantees in the absence of strong convexity in~\cite{YousefianNedic2019}. In convex nonsmooth (but smoothable) regimes, regularized smoothed SQN schemes, facilitated by a regularized smoothed limited memory BFGS scheme (rsL-BFGS) were presented in ~\cite{jalilzadeh2022variable}. 
	{(ii) \em Nonconvex regimes.} In recent work, Wang et al.~\cite{goldfarb} presented an SQN method for  unconstrained minimization of expectation-valued, smooth, and nonconvex functions. 
	In~\cite{chen2019stochastic}, a damped regularized L-BFGS
	algorithm for smooth nonconvex problems is developed.   {(iii) \em Nonsmooth and nonconvex settings.} SQN-type schemes in the structured nonsmooth convex setting, where the objective
	function can be separated into a smooth nonconvex part and a nonsmooth convex
	part, have been examined in~\cite{Wang2019StochasticPQ}
	and~\cite{yang2019stochastic}. Further in~\cite{CurtQue15},  an SQN scheme
	reliant on gradient sampling has been developed for addressing
unconstrained nonsmooth nonconvex optimization. (iv) {\em SQN with box
constraints.} A limited memory scheme that can handle box-type constraints was
proposed in \cite{byrd1995}. Variants of this scheme in nonconvex and nonsmooth
settings have been examined in \cite{Keskar2019}. {(v) \em SQN via zeroth-order
information.} The work on zeroth-order SQN methods has been far more
limited.  In~\cite{bollapragada2023adaptive}, an adaptive sampling zeroth-order
SQN method is developed but almost-sure convergence guarantees in the
stochastic, nonsmooth, nonconvex, constrained regime are not provided.  We
observe that recent work in ~\cite{berahas2016,goldfarb} requires
twice-differentiability assumptions which our smoothed function may not
satisfy. More recent research on zeroth-order schemes
in~\cite{bollapragada2023adaptive} requires that $f \in \mbox{C}^1$ but
provide eigenvalue bounds where the dependence on problem parameters is less
apparent. Our scheme requires neither smoothness nor convexity of $f$ and
provides eigenvalue bounds that bear some commonality with those
in~\cite{goldfarb} but with some key distinctions; specifically, our results accommodate nonsmooth and nonconvex integrands, while  the dependence on $L_0$ and $\eta$ is clarified in the resulting eigenvalue bounds. 
	A summary of the related work is presented in Tables~\ref{table:nonsmoothnonconvex} and~\ref{table:SQN}. 
		
		\smallskip
		
\noindent {\em Motivation.} Our work draws motivation from the absence of
efficient stochastic gradient and quasi-Newton schemes for resolving stochastic optimization problems when the
integrand is both nonsmooth and nonconvex. Our framework is inspired by
recent work by Zhang et al.~\cite{zhang20complexity}, where the authors show
that for a suitable class of nonsmooth functions, computing an
$\epsilon$-stationary point of $f$, denoted by $\x$, is impossible in finite
time where $\x$ satisfies $ \min \, \left\{ \, \|g\| \mid g \, \in \, \partial
f(\x) \, \right\} \, \leq \, \epsilon.$ This negative result is a consequence
of the possibility that the gradient can change in an abrupt fashion, thereby
concealing a stationary point. To this end, they introduce a notion of
$(\delta,\epsilon)$-stationarity, a weakening of $\epsilon$-stationarity;
specifically, if $\x$ is $(\delta,\epsilon)$-stationary, then there exists a
convex combination of gradients in a $\delta$-neighborhood of $\x$ with norm at
most $\epsilon$. However, this does not imply that $\x$ is $\delta$-close to
an $\epsilon$-stationary point of $\x$ as noted by
Shamir~\cite{shamir21nearapproximatelystationary} since the convex hull might
contain a small vector without any of the vectors being necessarily small.
However, as Shamir observes, one needs to accept that the
$(\delta,\epsilon)$-stationarity notion may have such pathologies. Instead, an
alternative might lie in minimizing a smoothed function $f_{\eta}$, defined as
$f_{\eta} (\x) = \mathbb{E}_{\bf u}[\, f(\x+{\eta} {\bf u})\, ]$ where ${\bf u}$ is
a suitably defined random variable. In fact, a Clarke-stationary point of
$f_{\eta}$ can be shown to be an $\eta$-Clarke stationary point (in terms of the enlarged
Clarke subdifferential)~\cite{goldstein77}.  This avenue allows for leveraging
a richer set of techniques but may still be afflicted by similar challenges.
		
		\begin{tcolorbox}
			{\bf Goal.} {\em Development of zeroth-order smoothing approaches in constrained and
				stochastic regimes with a view towards developing
				finite-time and asymptotic guarantees for stochastic gradient and SQN schemes.}
		\end{tcolorbox}
		
		\smallskip
		
		\noindent {\em Outline and contributions.} We develop a zeroth-order framework
		in regime where $f$ is expectation-valued over a closed and convex set $\Xscr$,
		where locally randomized smoothing is carried out via spherical smoothing. This
		scheme leads to zeroth-order stochastic gradient (\texttt{VRG-ZO}) and
		SQN (\texttt{VRSQN-ZO}) frameworks, where the gradient estimator is
		constructed via (sampled) function values.  After providing  some preliminaries
		on the notion of stationarity and randomized smoothing
		in~\cref{sec:preliminaries}, our main algorithmic contributions are
		captured in~\cref{sec:VRG_ZO} and~\cref{sec:SQN}, while preliminary numerics
		and concluding remarks are provided in~\cref{sec:numerics} and~\cref{sec:conc},
		respectively. Next, we briefly summarize our key contributions.
		
\smallskip
	
		\noindent (A) \underline{\texttt{VRG-ZO} schemes}.  When \texttt{VRG-ZO} is applied on  an
		$\eta$-smoothed counterpart of \eqref{eqn:prob}, where $\eta>0$ is a smoothing parameter, under
		suitable choices of the steplength and mini-batch sequences, we show  that the
		norm of the residual function of the smoothed problem tends to zero almost
		surely.  In particular, we show that the expected
		squared residual (associated with the smoothed problem) diminishes at the rate
		of $\mathcal{O}(1/k)$ in terms of projection steps on $\Xscr$, leading to
		an iteration complexity of \vrgzoIterationComplexity~and a 
		complexity of \vrgzoSampleComplexity~in terms of sampled
		function evaluations, where $\epsilon$ denotes an arbitrary bound on the
		expected norm of the residual of the $\eta$-smoothed problem.

\smallskip	
		\noindent (B) \underline{\texttt{VRSQN-ZO} schemes.} When
		\texttt{VRSQN-ZO} is applied on an unconstrained $\eta$-smoothed
		counterpart of \eqref{eqn:prob}, where the indicator function associated with
		the convex set $\Xscr$ is replaced by its Moreau-smoothed counterpart, under
		suitable choices of the steplength and mini-batch sequences, we show  that the
		norm of the residual of the smoothed problem tends to zero almost surely, leading to  iteration and
		sample complexities of \vrsqnzoIterationComplexity {~and}
		\vrsqnzoSampleComplexity, respectively.
		
				
				\smallskip {\bf Notation.} We use $\x$, $\x^\top$, and $\|\x\|$ to denote a column vector,
				its transpose, and its Euclidean norm, respectively. We
				define $f^*\triangleq {\displaystyle \inf_{\x \in \Xscr}} f(\x)$ and $f^*_\eta\triangleq
				{\displaystyle \inf_{\x \in \Xscr}} f_\eta(\x)$, where $f_\eta(\x)$ denotes the smoothed
				approximation of $f$. Given a continuous function, i.e., $f \in C^{0}$,
				we write $f \in C^{0,0}(\Xscr)$ if $f$ is Lipschitz on 
				$\Xscr$ with parameter $L_0$. Given a continuously differentiable
				function, i.e., $f \in C^{1}$, we write $f \in C^{1,1}(\Xscr)$ if $\nabla f$ is
				Lipschitz continuous on $\Xscr$ with parameter $L_1$. We
				write a.s. for ``almost surely” and {$\mathbb{E}[h(\x,\bxi)]$} denotes the
				expectation {with respect to a random variable $\bxi$}. 
				$\Pi_{\Xscr}[\vv]$ denotes the Euclidean projection of $\vv$ onto the closed convex set $\Xscr$.
				
				\section{Stationarity and smoothing}
				\label{sec:preliminaries}
				We first recap some concepts of Clarke's nonsmooth calculus~\cite{clarke98}
				that allow for providing stationarity conditions.  The directional derivative, crucial in addressing nonsmooth optimization problems, is defined next. 
				\begin{definition}[{\bf Directional derivatives and Clarke generalized gradient~\cite{clarke98}}] \em 
					The directional derivative of $h$ at $\x$ in a direction $v$ is defined as 
					\begin{align}
						h^{\circ}(\x,v) \triangleq  \limsup_{\y \to \x, t \downarrow 0} \left(\frac{h(\y+tv)-h(\y)}{t}\right).
					\end{align}
					The Clarke generalized gradient of $h$ at $\x$ can then be defined as 
					\begin{align}
						\partial h(\x) \, \triangleq \, \left\{ \, {\zeta} \, \in \, \Real^n\, \mid\, h^{\circ}(\x,v) \, \geq \,  {\zeta}^\top v, \quad \forall v \, \in\, \Real^n\,\right\}.
					\end{align}
					In other words, $h^{\circ}(\x,v) = \displaystyle \sup_{g \in \partial h(\x)}  g^{\top}v.$ $\hfill \Box$
				\end{definition}
				If $h$ is $C^1$ at $\x$, the Clarke generalized gradient reduces to the standard gradient, \fy{i.e.,} $\partial h(\x) = \nabla_{\x} h(\x).$ If $\x$ is a local minimizer of $h$, then we have that $0 \in \partial h(\x)$. In fact, this claim can be extended to convex constrained regimes, \fy{i.e.,} if $\x$ is a local minimizer of 
				$\displaystyle \min_{\x \in \Xscr} \ h(\x)$,
				then $\x$ satisfies $0 \in \partial h(\x) + \Nscr_{\Xscr}(\x)$, 
				where $\Nscr_{\Xscr}(\x)$ denotes the normal cone of $\Xscr$ defined at
				$\x$~\cite{clarke98}. Furthermore, if
				$h$ is locally Lipschitz on an open set \fyy{$\mathcal{C}$} containing $\Xscr$, then $h$ is
				differentiable almost everywhere on $\Cscr$ by Rademacher's
				theorem~\cite{clarke98}. Suppose \fyy{$\Cscr_h$} denotes the set of points where $h$
				is not differentiable. We now provide some properties of the Clarke generalized
				gradient.   \begin{proposition}[{\bf Properties of Clarke generalized
						gradients~\cite{clarke98}}] \em
					Suppose $h$ is $L_0$-Lipschitz continuous on $\Real^n$. Then the following hold \fyy{for any $x \in \Real^n$}.
					\begin{enumerate}
						\item[(i)] $\partial h(\x)$ is a nonempty, convex, and compact  set and $\|g \| \leq {L_0}$ for any $g \in \partial h(\x)$. 
						\item[(ii)] $h$ is differentiable almost everywhere.  (iii) $\partial h(\x)$ is an upper-semicontinuous map defined as 
						$\partial h(\x) \, \triangleq \, \mbox{conv}\left\{\, g \, \mid \, g \, = \, \lim_{k \to \infty} \nabla_{\x} h(\x_k), \Cscr_h \, \not \, \owns \x_k \, \to \, \x\, \right\}.$ $\Box$
					\end{enumerate}
				\end{proposition}
				We may also define the $\delta$-Clarke generalized gradient~\cite{goldstein77} as  
				\begin{align}
					\partial_{\delta} h(\x) \, \triangleq \, \mbox{conv}\left\{ \, {\zeta}\, \mid \, {\zeta} \, \in \, \partial h(\y), \|\x-\y\|\, \leq \, \fyy{\delta} \,\right\}.
				\end{align}
				Notice that $\delta$-Clarke generalized gradient of $\x$ is given by the convex hull of elements in the Clarke generalized gradient of vectors within a distance $\delta$ of $\x$. When $h$ is a nonsmooth and nonconvex function, 
				we address nonsmoothness by considering a locally randomized smoothing technique. Given a function $h:\mathbb{R}^n \to \mathbb{R}$ and a scalar $\eta>0$, a spherically smoothed approximation of $h$, denoted by $h_\eta$, is defined as 
				\begin{align}\label{def-smooth}
					h_\eta (\x) \, \triangleq \, \mathbb{E}_{\uu \in \mathbb{B}}\left[\, h(\x+\eta \uu)\, \right],
				\end{align} where {$\uu$ denotes a random variable uniformly distributed on a unit ball $\mathbb{B}$, defined as  $\mathbb{B}\triangleq \{u \in \mathbb{R}^n\mid \|u\|\leq 1\}$, with realizations denoted by $u \in \mathbb{B}$}.
				We now recall a concentration inequality and subsequently present properties for spherical smoothing, where
				$\mathbb{S}\triangleq \{v \in \mathbb{R}^n \mid
				\|v\|=1\}$ is the surface of $\mathbb{B}$.
				\begin{lemma}[L\'{e}vy concentration on $ \mathbb{S}^n $~{\cite[Prop. 3.11 and Example 3.12]{Wainwright2019}}]
					\label{lemma:levy-concentration}\em
					Let ${h}: \mathbb{R}^n \to \mathbb{R}$ be $L$-Lipschitz and $\tilde{\vv}$ be uniformly distributed on {$\mathbb{S}$}. Then 
					\begin{align}
						\mathbb{P}\left[\, \left|\, h(\tilde{\vv}) - \mathbb{E}[h(\tilde{\vv})]\, \right|\, \geq \, \alpha \, \right] &\, \leq \, 2 \sqrt{2\pi} e^{-\frac{n \alpha^2}{8L^2}}.
					\end{align}
				\end{lemma}	
				
\begin{lemma}[{\bf Properties of spherical smoothing}] \label{lemma:props_local_smoothing}\em Suppose $h:\mathbb{R}^n \mapsto \mathbb{R}$ is {a continuous function and its smoothed counterpart $h_{\eta}$ is defined as \eqref{def-smooth}, where $\eta>0$ is a given scalar}. Then the following hold for any $\x, \y \in {\cal X}$.
					
\begin{enumerate}
						\item[ (i)] $h_{\eta}$ is $C^1$ over $\Xscr$ and $\nabla_{\x} h_{\eta}(\x) = \left(\tfrac{n}{2\eta}\right) \mathbb{E}_{\vv \in \eta \mathbb{S}} \left[\left(h(\x+\vv)-h(\x-\vv)\right) \tfrac{\vv}{\|\vv\|}\right]. $ 
					\end{enumerate}
Suppose $h \in C^{0,0}(\Xscr_{\eta})$ with parameter $L_0$. For any $\x, \y \in \Xscr$,  the following hold. 
					
					\begin{enumerate}
						\item[ (ii)] $| h_{\eta}(\x)-h_{\eta}(\y) | \leq L_0 \|\x-\y\|.$  (iii) $| h_{\eta}(\x) - h(\x)| \leq L_0 \eta.$ 
						
						\item[ (iv)] $\| \nabla_{\x} h_{\eta}(\x) -\nabla_{\x} h_{\eta}(\y)\| \leq {
							\tfrac{L_0 \sqrt{n}}{\eta}}\|\x-\y\|.$
						
						
						
						\item [(v)] {Suppose $h \in C^{0,0}(\Xscr_{\eta})$ with parameter $L_0$ {and} 
							$g_{\eta}(\x,\vv)  \triangleq \left(\tfrac{n(h(\x+\vv) - h(\x-\vv))\vv}{2\eta\|\vv\|}\right)$
							for $\vv \in \eta\mathbb{S}$.  Then for} any $\x \in \Xscr$, we have that ${\mathbb{E}_{\vv \in \eta \mathbb{S}}}[{\|{g_\eta}(\x,\vv)\|^2}] \leq 16\sqrt{2\pi}L_0^2 n$. $\hfill \Box$ 
					\end{enumerate}
				\end{lemma}
				\begin{proof}
					{(i) By Stoke's theorem and spherical symmetry of $\vv$, we have that 
						\begin{align*}
							\nabla_{\x} h_{\eta}(\x)& = \tfrac{n}{\eta} \mathbb{E}_{\vv \in \eta \mathbb{S}} \ \left[ \tfrac{h(\x+\vv) \vv }{\|\vv\|} \right] = 
							\tfrac{n}{2\eta} \mathbb{E}_{\vv \in \eta \mathbb{S}} \ \left[ \tfrac{(h(\x+\vv)+h(\x+\vv)) \vv }{\|\vv\|} \right] \\
							& = \tfrac{n}{2\eta} \mathbb{E}_{\vv \in \eta \mathbb{S}} \ \left[ \tfrac{(h(\x+\vv)) \vv }{\|\vv\|}+\tfrac{(h(\x+\vv)) \vv }{\|\vv\|} \right] = 
							\tfrac{n}{2\eta} \mathbb{E}_{\vv \in \eta \mathbb{S}} \ \left[ \tfrac{(h(\x+\vv)) \vv }{\|\vv\|}+\tfrac{(h(\x-\vv)) (-\vv) }{\|\vv\|} \right] \\
							& = \tfrac{n}{2\eta} \mathbb{E}_{\vv \in \eta \mathbb{S}} \ \left[ \tfrac{((h(\x+\vv))-(h(\x-\vv)) \vv) }{\|\vv\|} \right]. 
						\end{align*}
						(ii,iii) follow akin to ~\cite[Lemma 1]{CSY2021MPEC}. 
						(iv)
Recall from~\cite[Prop. 2.2]{lin2022} that $\|
\nabla_{\x} h_{\eta}(\x) -\nabla_{\x}
h_{\eta}(\y)\| \leq \tfrac{C_n
L_0}{\eta}\|\x-\y\|,$ for $n\geq 1$, where
$C_n \triangleq
\left(\tfrac{2}{\pi}\right)^{p_n}
\tfrac{n!!}{(n-1)!!}$ and $p_n = 1$ if $n$
is even, and $p_n =0$ otherwise. It suffices to show that $C_n \le 1$ for $n \ge 1$. For $n \ge
0$, let $w_n$ denote the Wallis'
integrals~\cite{spivak2006calculus}, given as
$w_n = \int_{0}^{\pi/2}\sin^n(u)du$. Then we
have $w_n =\left(\tfrac{n-1}{n}\right)w_{n-2}$ for $n \geq
2$ where $w_0=\tfrac{\pi}{2}$ and $w_1=1$
(cf.~\cite[Ch. 19]{spivak2006calculus}). This
implies (via mathematical induction) that
$w_n =
\left(\tfrac{\pi}{2}\right)\tfrac{(n-1)!!}{n!!}$
for even $n$, while $w_n =
\tfrac{(n-1)!!}{n!!}$ for odd $n$. Let us
define $v_n \,\triangleq \, (n+1)w_n w_{n+1}$ for $n\geq 0$.
We obtain $v_n =
(n+1)\left(\tfrac{\pi}{2}\right)
\tfrac{(n-1)!!}{n!!}\tfrac{n!!}{(n+1)!!} =
\tfrac{\pi}{2}$, implying that $v_n =
\tfrac{\pi}{2}$ for all $n\geq 0$. Recall
that from~\cite{spivak2006calculus} we have
$w_{n+1} < w_n$ for all $n\geq 0$. Thus we
obtain $\tfrac{\pi}{2} = v_n  =(n+1)w_n
w_{n+1} < (n+1)w_n^2$ for all $n\geq 0$, i.e., $w_n < \tfrac{\pi}{2(n+1)}$ for $n \ge 0$.  Note
that $C_1 =1$. Suppose $n \geq 2$. We have
$C_n= \tfrac{1}{w_n} <
\sqrt{\tfrac{2(n+1)}{\pi}} =
\sqrt{\tfrac{2(1+1/n)}{\pi}} \sqrt{n} \leq
\sqrt{\tfrac{2(1+1/2)}{\pi}} \sqrt{n} \leq
0.98 \sqrt{n} $. This implies that $C_n  \leq
1$ for all $n\geq 1$.} {(v) We have that}
\begin{align*}
							\mathbb{E}\left[ \|g_{\eta}(\x,\vv)\|^2 \right] & \le 
							\tfrac{n^2}{4\eta^2} \mathbb{E}\left[ \|\tfrac{(h(\x+\vv)-h(\x-\vv))\vv}{\|\vv\|}\|^2 \right] \le 
							\tfrac{n^2}{4\eta^2} \mathbb{E}\left[  (h(\x+\vv)-h(\x-\vv))^2 \tfrac{\|\vv\|^2}{\|\vv\|^2} \right] \\
							& \le 
							\tfrac{n^2}{4\eta^2} \mathbb{E}\left[ \left(h(\x+\vv)-\mathbb{E}\left[h(\x+\vv)\right] + \mathbb{E}\left[ h(\x+\vv)\right]-h(\x-\vv)\right)^2  \right]. 
						\end{align*}
						But $\mathbb{E}\left[ h(\x+\vv)\right] = \mathbb{E}\left[ h(\x-\vv)\right]$ by spherical symmetry. Consequently, 
						\begin{align*}
							\mathbb{E}\left[ \|g_{\eta}(\x,\vv)\|^2 \right] & \le 
							\tfrac{n^2}{4\eta^2} \mathbb{E}\left[  \left(h(\x+\vv)-\mathbb{E}\left[h(\x+\vv)\right] + \mathbb{E}\left[ h(\x-\vv)\right]-h(\x-\vv)\right)^2  \right] \\
							& \le \tfrac{2n^2}{4\eta^2} \mathbb{E}\left[ \left|(h(\x+\vv)-\mathbb{E}\left[h(\x+\vv)\right]\right|^2  + \left|\mathbb{E}\left[ h(\x-\vv)\right]-h(\x-\vv))\right|^2 \right]. 
						\end{align*}
						Then by invoking L\'{e}vy's inequality (cf. Lemma~\ref{lemma:levy-concentration}), we have that 
						\begin{align*}
							\mathbb{E}\left[ |(h(\x+\vv)-\mathbb{E}\left[h(\x+\vv)\right]|^2 \right] \le 2\sqrt{2\pi}\int_0^{\infty} e^{-\alpha n}{8 \eta^2 L_0^2} d\alpha = \tfrac{16\sqrt{2\pi} \eta^2 L_0^2}{n}. 
						\end{align*}
					{By noting that $\mathbb{E}\left[ \left|(h(\x+\vv)-\mathbb{E}\left[h(\x+\vv)\right]\right|^2  \right] = \mathbb{E}\left[  \left|\mathbb{E}\left[ h(\x-\vv)\right]-h(\x-\vv))\right|^2 \right]$ and leveraging the preceding bound, we get}  
						$\mathbb{E}\left[ \|g_{\eta}(\x,\vv)\|^2 \right] \le \tfrac{n^2}{\eta^2} \tfrac{16\sqrt{2\pi} \eta^2 L_0^2}{n} = 16\sqrt{2\pi}L_0^2 n.$
					
				\end{proof}
In developing \texttt{VRG-ZO}, we minimize an $L_0$-Lipschitz
function on $\Xscr_{\eta_0}$, where  $\Xscr_{\eta_0} \triangleq
\Xscr + \eta_0 \mathbb{B}$ represents the Minkowski sum of $\Xscr$
and $\eta \mathbb{B}$.  Further, we assume that the random
variable $\f(\x,\bxi) - f(\x)$ admits suitable bias and moment
properties.  We develop schemes for computing approximate
stationary points of \eqref{eqn:prob} by an iterative scheme.
However, we formalize the relationship between the original problem
and its smoothed counterpart.  This is provided
in~\cite{CSY2021MPEC, lin2022} by leveraging results
from~\cite{mayne84,Mordukhovich2006}.
				\begin{proposition}[{\bf Stationarity of $f_{\eta} \Rightarrow \eta$-Clarke
						stationarity of $f$}]\label{prop_equiv} \em 
					Consider \eqref{eqn:prob} where $f$ is a locally Lipschitz continuous function and $\Xscr$ is a closed, convex, and bounded set in $\Real^n$. 
					\noindent (i) For any $\eta > 0$ and any $\x \in \Real^n$, $\nabla f_{\eta}(\x) \in \partial_{\eta} f(\x)$. Furthermore, if $0 \not \in \partial f(\x)$, then there exists an $\eta$ such that $\nabla_{\x} f_{\tilde \eta} (\x) \neq 0$ for $0 < \tilde{\eta} \le \eta$.    
					\noindent (ii) For any $\eta > 0$ and any $\x \in \Xscr$, if $0 \in \nabla_{\x} f_{\eta}(\x) + \mathcal{N}_{\Xscr}(\x)$, then  $0 \in \partial_{\eta} f(\x)+ \mathcal{N}_{\Xscr}(\x)$. 
				\end{proposition}
				Intuitively, this means that if $\x$ is a stationary point of the
				$\eta$-smoothed problem, then $\x$ is an $\eta$-Clarke stationary point of the
				original problem. {We emphasize that the ``$\eta$'' in ``{$\eta$-Clarke} stationary'' refers to the size of the augmentation of the generalized Clarke gradient but is not a bound on the norm of the elements in the Clarke generalized gradient, as is sometimes suggested with terms like ``$\eta$-stationary'' or ``$\eta$-optimal.''} Next, we introduce a residual function that captures the
				departure from stationarity~\cite{beck23introduction}. Recall that when $h$ is a differentiable but
				possibly nonconvex function and $\Xscr$ is a closed and convex set, then  $\x$
				is a stationary point of \eqref{eqn:prob} if and only if 
				$$ G_{\beta}(\x) \triangleq \beta\left( \x - \Pi_{\Xscr}\left[\x - \tfrac{1}{\beta} \nabla_{\x} f(\x) \right]\right) = 0. $$
				When $f$ is not necessarily smooth as considered here, a residual of the
				smoothed problem can be derived by replacing $\nabla_{\x} f(\x)$ by
				$\nabla_{\x} f_{\eta}(\x)$. In particular, the residual $G_{\eta,\beta}$
				denotes the stationarity residual with parameter $\beta$ of the $\eta$-smoothed
				problem while $\tilde{G}_{\eta,\beta}$ is its counterpart arising from using an
				error-afflicted estimate of $\nabla_{\x} f_{\eta}(\x)$.  
				\begin{definition}[{\bf The residual mapping}]\label{def:res_maps}\em Given $\beta>0$ and a smoothing parameter $\eta>0$, for any $\x \in \mathbb{R}^n$ and an arbitrary vector $\tilde e \in \mathbb{R}^n$,
					let the residual mappings $G_{\eta,\beta}(\x)$ and $\tilde{G}_{\eta,\beta}(\x, {\tilde e})$ be defined as $G_{\eta,\beta}(\x)  \triangleq \beta {\left(\x - \Pi_{\Xscr}\left[\x - \tfrac{1}{\beta} \nabla_x {f_{\eta}(\x)}\right]\right)}$ and $\tilde{G}_{\eta,\beta}(\x, {\tilde e})  \triangleq \beta\left( \x - \Pi_{\Xscr}\left[\x - \tfrac{1}{\beta} ({\nabla_x f_{\eta}(\x)}+\tilde{e}) \right]\right)$, respectively.
				\end{definition}
				This allows for deriving a bound on $G_{\eta,\beta}(\x)$ in terms of $\tilde{G}_{\eta,\beta}(\x, {\tilde e})$ and $\tilde{e}$ as follows.
				\begin{lemma}[{\cite[Lemma 6]{CSY2021MPEC}}]\label{lem:inexact_proj_2}\em
					For any $\beta,\eta>0, \x \in \Xscr, \tilde{e} \in\mathbb{R}^n$, 
					$$\|G_{\eta,\beta}(\x) \|^2   \leq  {2} \| \tilde{G}_{\eta,\beta}(\x, {\tilde e}) \|^2 + 2 \|\tilde{e}\|^2.$$  
				\end{lemma}
				Next, we recall a result that relates the residual function as $\gamma$ changes~\cite{beck23introduction}.
				\begin{lemma}\label{lemma:mon_G} \em Suppose $\eta > 0$ and $\{\gamma_k\}$ is a diminishing sequence. Then, for any $\x$ and $k \geq 0$, $\| G_{\eta,1/\gamma_k} (\x)\| \, \leq \,\| G_{\eta,1/\gamma_{k+1}} (\x)\|$. 
				\end{lemma}
				We use the following result in establishing the almost sure convergence guarantees.
				\begin{lemma}[Robbins-Siegmund Lemma]\label{lemma:Robbins-Siegmund}\em Let $v_k,$ $u_k,$ $\alpha_k,$ and  $\beta_k$ be
					nonnegative random variables, and let the
					following relations hold almost surely. 
					\[\EXP{v_{k+1}\mid {\tilde \sF_k} } 
					\le (1+\alpha_k)v_k - u_k + \beta_k \quad\hbox{ for all } k,\qquad
					\sum_{k=0}^\infty \alpha_k < \infty,\qquad
					\sum_{k=0}^\infty \beta_k < \infty,\] 
					where $\tilde \sF_k$ is the collection $v_0,\ldots,v_k$, $u_0,\ldots,u_k$,
					$\alpha_0,\ldots,\alpha_k$, $\beta_0,\ldots,\beta_k$. 
					Then ${\displaystyle \lim_{k\to\infty}}v_k = v$ a.s. and $ \sum_{k=0}^\infty u_k < \infty$ a.s., where $v \geq 0$ is some random variable.
				\end{lemma}
				
				\begin{algorithm}[hbt]
					\caption{\texttt{{VRG-ZO}}: Variance-reduced randomized zeroth-order {gradient} method}\label{algorithm:zo_nonconvex}
					{   \begin{algorithmic}[1]
							\STATE\textbf{input:}  Given $\x_0 \in \Xscr$, stepsize $\gamma>0$, smoothing parameter $\eta>0$, mini-batch sequence $\{N_k\}$, and an integer $R_{\ell,K}$ randomly selected from $\{\ell\triangleq\lceil\lambda K\rceil ,\ldots,K\}$ using a discrete uniform distribution where $\lambda \in (0,1)$ and $K\geq 1$.
							\FOR {$k=0,1,\ldots,{K}-1$}
							\STATE   Generate a random {mini-}{batch $v_{j,k} \in \eta \mathbb{S}$ for $j=1,\ldots,N_k$}  
							\STATE  Compute $\{g_{\eta}(\x_k,v_{j,k},\xi_{j,k})\}_{j=1}^{N_k}$, where for all $j=1,\ldots,N_k$,\\ $g_{\eta}(\x_k,v_{j,k},\xi_{j,k}) :=\left(\tfrac{n}{\us{2}\eta}\right)\left({\f}(\x_k+ v_{j,k},\xi_{j,k}) - {\f}(\x_k- v_{j,k},\xi_{j,k})\right)\tfrac{v_{j,k}}{\left\|v_{j,k}\right\|}$ 
							\STATE   {Compute} mini-batch inexact {ZO} gradient. $
							g_{\eta,N_k}(\x_k) := \tfrac{\sum_{j=1}^{N_k}  g_{\eta}\left(\x_k,v_{j,k},\xi_{j,k}\right)}{N_k}$
							\STATE   {Update $\x_k$ as follows.}
							$\x_{k+1}:= \Pi_{\Xscr}\left[ \x_{k}-\gamma g_{\eta,N_k}(\x_k) \right]$
							\ENDFOR
							\STATE Return $\x_{R_{\ell,K}}$ 
					\end{algorithmic}}
				\end{algorithm}
				\section{A randomized zeroth-order gradient method}\label{sec:VRG_ZO}
				In this section, we provide the main assumptions, outline the proposed zeroth-order method, derive some preliminary results, and present our convergence analysis. 
\subsection{Preliminaries}

\begin{assumption}[{\bf Problem properties}] 
		\label{ass-1}\em
		\noindent Consider problem \eqref{eqn:prob}.
		
		\noindent (i) {$\tilde{f}(\bullet,\bxi)$ is $\tilde{L}_0(\bxi)$}-Lipschitz on $\Xscr + \eta_0 \mathbb{B}$ for some $\eta_0 > 0$, a.s. for all $\bxi$,  where $L_0\triangleq \sqrt{\mathbb{E}[{\tilde{L}_0}(\bxi)^2]} <\infty$.  
		\noindent (ii) $\Xscr \subseteq \Real^n$ is a nonempty, closed, and convex set.
		\noindent (iii) For any $\x \in \Xscr+ \eta_0 \us{\mathbb{B}}$, $\mathbb{E}[\, {\f(\x,\bxi)} \, \mid \,  \x\, ] \, =\, f(\x)$.
\end{assumption}
				We introduce a variance-reduced randomized zeroth-order scheme presented by
				Algorithm \ref{algorithm:zo_nonconvex} and define a zeroth-order gradient
				estimate of  $\f(\x_k, \xi_{j,k})$ as 
				\begin{align*}
					g_{\eta}(\x_k,v_{j,k},\xi_{j,k}) \triangleq \left(\tfrac{n}{\eta}\right)\left({\f}(\x_k+ v_{j,k},\xi_{j,k}) - {\f}({\x_k-v_{j,k}},\xi_{j,k})\right)\tfrac{v_{j,k}}{\left\|v_{j,k}\right\|},
				\end{align*}
				where $v_{j,k} \in \eta\mathbb{S}$.  {Intuitively, $g_{\eta}(\x_k,v_{j,k},\xi_{j,k})$ generates a gradient estimate by employing sampled function evaluations $\tilde{f}(\x_k{+v_{j,k}},\xi_{j,k})$ and $\tilde{f}(\x_k{-v_{j,k}},\xi_{j,k})$; in short, the zeroth-order oracle, given an $\x_k$ and a perturbation vector $v_{j,k}$, produces two evaluations, \fy{i.e.,}   
					$\tilde{f}(\x_k +v_{j,k},\xi_{j,k})$ and $\tilde{f}(\x_k-v_{j,k},\xi_{j,k})$.}
				We formally define the stochastic errors emergent from the randomized scheme based on a mini-batch approximation.
				\begin{definition}\label{def:stoch_errors} \em
					For all $k\geq 0,j=1,\ldots,N_k$, let 
					$\mathbf{e}_k\triangleq \tfrac{\sum_{j=1}^{N_k}\mathbf{e}_{j,k}}{N_k}$, 
					\begin{align}
						e_{j,k} & \triangleq {g_{\eta}(\x_k,v_{j,k},\xi_{j,k}) -\nabla f_\eta(\x_k)}, \ \mbox{and } {\mathbf{e}_{j,k}  \triangleq g_{\eta}(\x_k,\vv_{j,k},\bxi_{j,k})}-\nabla f_\eta(\x_k)\label{def:stoch_errors1}. \mbox{$\ \Box$}
					\end{align}
				\end{definition}
				We employ boldface  to denote random variables (such as $\mathbf{e}_{j,k}$)  while their realizations are not emboldened (such as $e_{j,k}$).
				The history of Algorithm \ref{algorithm:zo_nonconvex} at iteration $k$ is denoted by $\mathcal{F}_k$, defined as
				\begin{align}\label{eqn:history}
					\mathcal{F}_0 \triangleq \{x_0\},\quad 
					\mathcal{F}_k \triangleq \mathcal{F}_{k-1} \, \bigcup \, \left\{ \, \bigcup_{j=1}^{N_{k-1}}\, \{\xi_{j,k-1}, v_{j,k-1}\} \, \right\}, \qquad \hbox{for } k \geq 1.
				\end{align} 
				We require independence between $\xi_{j,k}$ and $v_{j,k}$, where $v_{j,k}$ is user-defined.
				\begin{assumption}[{\bf Independence}]\label{assum:random_vars}\em
					Random samples $\xi_{j,k}$, $v_{j,k}$ are generated independent of each other for all $k\geq 0$ and $1\leq j \leq N_k$.
				\end{assumption}
				We now analyze the bias and moment properties of {the error sequence}.
				\begin{lemma}[{\bf Bias and moment properties of ${\bf e}$}] \label{lem:stoch_error_var}\em
					Let Assumptions \ref{ass-1} and \ref{assum:random_vars} hold. Then (i) and (ii) hold a.s. for $k\geq 0$, $N_k \geq 1$, and $j = 1, \hdots, N_k$.
					
					\noindent (i) $\mathbb{E}[\mathbf{e}_{j,k}\mid \mathcal{F}_k] = 0$;  \noindent (ii) $\mathbb{E}[\|\mathbf{e}_{k}\|^2\mid \mathcal{F}_k] \leq  \tfrac{16\sqrt{2\pi} L_0^2 n}{N_k}$.
				\end{lemma}
				\begin{proof}
						(i) To show the conditional unbiasedness of $\mathbf{e}_{j,k}$, we proceed as follows.
						\begin{align*}
							\mathbb{E}_{\bxi,\! \vv}\left[ {g_{\eta}(\x_k,\!\vv_{j,k},\!\bxi_{j,k})}  \mid {\cal F}_k\right]\! & =\mathbb{E}_{\vv}\!   
							\left[\mathbb{E}_{\bxi}\left[ \tfrac{n(\f(\x_k+v_{j,k},\xi_{j,k})-\f(\x_k-v_{j,k},\xi_{j,k}))v_{j,k}}{2\eta \|v_{j,k}\|} \hspace{-0.05in} \mid {\cal F}_k \cup \{v_{j,k}\} \right] \right] \\
							&\hspace{-1in}  \overset{\tiny \mbox{Lemma~\ref{lemma:props_local_smoothing}~{(i)}}}{=}
							\mathbb{E}_{\vv}\left[ \tfrac{n}{2\eta}\tfrac{(f(\x_k+v_{j,k})-f(\x_k\lm{-}v_{j,k}))v_{j,k}}{\|v_{j,k}\|} \mid {\cal F}_k \right]  = \nabla_{\x} f_{\eta}(\x_k).    
						\end{align*} 
						By leveraging Lemma~\ref{lemma:props_local_smoothing}~{(v)}, we have that 
						\begin{align*}
							\mathbb{E}_{{\bxi, \vv}}\left[\, \|\mathbf{e}_{j,k}\|^2 \mid {\cal F}_k \, \right] &= 
							\mathbb{E}_{{\bxi}}\left[\, \mathbb{E}_{ {\vv}}\left[\, \|\mathbf{e}_{j,k}\|^2  \, \mid{\cal F}_{k}\cup\{\xi_{j,k}\} \right] \right]\\
							& \le 
							\mathbb{E}_{{\bxi}}\left[\, 16\sqrt{2\pi}\tilde{L}_0(\bxi_{{j,k}})^2n \right] \le 16\sqrt{2\pi} L_0^2 n. 
						\end{align*}
						It follows that $\mathbb{E}\left[\, \|\mathbf{e}_{k}\|^2  \, \mid {\cal F}_k\right] \le \tfrac{16\sqrt{2\pi} L_0^2 n}{N_k}.$
				\end{proof}
								\subsection{Convergence and rate analysis}
				We now derive convergence guarantees under  constant and diminishing stepsizes, beginning with the following result.   
				\begin{lemma}\label{lemma:zog_ineq}\em 
					Let Assumption \ref{ass-1} hold {and suppose} $\left\{\x_k\right\}$ is generated by
					Algorithm~\ref{algorithm:zo_nonconvex}. Let $\{\gamma_k\}$ be a
					non-increasing stepsize sequence (e.g., constant or diminishing) where
					$\gamma_0 \in (0,{\tfrac{\eta}{{\sqrt{n}}L_0}})$ for $\eta>0$. Then  the
					following hold. 
					
					\noindent (i) For $k\geq 0$, we have 
					\begin{align}\label{eqn:recursive_ineq_lemma5} 
						\left( 1-\tfrac{{\sqrt{n}}L_0\gamma_k}{\eta}\right)\tfrac{\gamma_k \|G_{\eta,1/\gamma_0} (\x_k)\|^2}{4}     &   \leq   f_{\eta}(\x_k)-      f_{\eta}(\x_{k+1}) 
						+{\left( 1-\tfrac{{\sqrt{n}}L_0\gamma_k}{2\eta}\right)} \gamma_k \|e_k\|^2.
					\end{align}
					
					\noindent (ii) Let $R_{\ell,K}$ be a random integer taking values in $\{\ell, \ldots, K-1\}$ with mass function $\mathbb{P}{[} R_{\ell,K} = j {]} = \tfrac{\gamma_j}{\sum_{i=\ell}^{K-1} \gamma_i}$ for any $j \in \{\ell, \ldots, K-1\}$.  Then, for all $k\geq 0$ we have 
					\begin{align}\label{ineq:residual_map_ineq_VRGZO_2}
						\mathbb{E}\left[ \|G_{\eta,{1/\gamma_0}} (\x_{R_{\ell,K}})\|^2\right]
						&\leq 
						\tfrac{\left( \mathbb{E}\left[f(\x_{\ell})\right] -f^* +2L_0\eta+ \left({16 \sqrt{2\pi} L_0^2{n} }\right)\sum_{k=\ell}^{K-1} \tfrac{\gamma_k}{N_k}\right)}{\left(\left( \tfrac{1}{4}-\tfrac{{\sqrt{n}} L_0\gamma_0}{4\eta}\right) \sum_{k=\ell}^{K-1} \gamma_k\right)} .
					\end{align}
				\end{lemma}
				
				\begin{proof}
					(i) From Lemma~\ref{lemma:props_local_smoothing} (iv), the gradient mapping $\nabla f_\eta$ is Lipschitz continuous with the parameter $L_{\eta}\triangleq \tfrac{\sqrt{n}L_0}{\eta}$. From the descent lemma, 
					\begin{align}\label{eqn:descent_lem_proof_eq1}
					\notag	f_{\eta}(\x_{k+1}) & \leq f_{\eta}(\x_{k}) + \nabla f_{\eta}(\x_k)^\top(\x_{k+1}-\x_k)
						+\tfrac{L_{\eta}}{2}\|\x_{k+1}-\x_k\|^2\\
					\notag	&=f_\eta(\x_k) + \left(\nabla f_{\eta}(\x_k)+{e_k}\right)^\top(\x_{k+1}-\x_k)\\
		& -{e_k}^\top(\x_{k+1}-\x_k)+\tfrac{L_{\eta}}{2}\|\x_{k+1}-\x_k\|^2.
					\end{align}
					Invoking 
						the properties of the Euclidean projection, we have \\$ 
					(\x_k-\gamma_k\left(\nabla f_\eta(\x_k)+e_k\right)-\x_{k+1})^\top(\x_k-\x_{k+1})\leq 0$. This implies that 
					\begin{align}\label{eqn:descent_lem_proof_eq2}
						\left(\nabla f_\eta(\x_k)+{e_k}\right)^\top(\x_{k+1}-x_k)
						\leq -\tfrac{1}{\gamma_k}\|\x_{k+1}-\x_k\|^2.
					\end{align}
					In addition, we may also express $e_k^{\top}(\x_{k+1}-\x_k)$ as follows. 
					\begin{align}\label{eqn:descent_lem_proof_eq3}
						-e_k^{\top}(\x_{k+1}-\x_k)
						\leq \tfrac{{\gamma_k}}{2}\|e_k\|^2+\tfrac{1}{2{\gamma_k}}\|\x_{k+1}-\x_k\|^2.
					\end{align}
					Combining the inequalities \eqref{eqn:descent_lem_proof_eq1}, \eqref{eqn:descent_lem_proof_eq2}, and \eqref{eqn:descent_lem_proof_eq3} we obtain 
					\begin{align*}
						f_{\eta}(\x_{k+1})  & \leq   f_{\eta}(\x_k)+ \left( -\tfrac{1}{2\gamma_k}+\tfrac{L_{\eta}}{2}\right) \|\x_{k+1}-\x_k\|^2 +\tfrac{\gamma_k}{2} \|e_k\|^2 \\
	& \overset{\tiny \mbox{Def.~\ref{def:res_maps}}}{=} f_{\eta}(\x_k)- \left( \tfrac{1}{2\gamma_k}-\tfrac{L_{\eta}}{2}\right)\gamma_k^2 \left\|\tilde{G}_{\eta,1/\gamma_k}(\x_k, {e_k})\right\|^2+\tfrac{\gamma_k}{2} \|e_k\|^2\\
	& = f_{\eta}(\x_k)- \left( 1-L_{\eta}\gamma_k\right)\tfrac{\gamma_k}{2} \left\|\tilde{G}_{\eta,1/\gamma_k}(\x_k, {e_k})\right\|^2+\tfrac{\gamma_k}{2} \|e_k\|^2\\
	& \hspace{-0.2in}\overset{\tiny \mbox{Lemma~\ref{lem:inexact_proj_2}}}{\le} f_{\eta}(\x_k)- \left( 1-L_\eta \gamma_k\right)\tfrac{\gamma_k}{4} \left\|G_{\eta,1/\gamma_k}(\x_k)\right\|^2+\left(\tfrac{\gamma_k}{2} + \left( 1-L_{\eta} \gamma_k\right)\tfrac{\gamma_k}{2}\right) \|e_k\|^2,
					\end{align*}
where $1-L_{\eta}\gamma_k > 0$ for any $k$  since $\gamma_0 < \tfrac{1}{L_{\eta}}$. 
					From the preceding relation and invoking Lemma~\ref{lemma:mon_G}, we obtain~\eqref{eqn:recursive_ineq_lemma5}, i.e.,
\begin{align*}
\left( 1-L_\eta \gamma_k\right)\tfrac{\gamma_k}{4} \left\|G_{\eta,1/\gamma_0}(\x_k)\right\|^2 \, \le \, f_{\eta}(\x_k)- f_{\eta}(\x_{k+1}) + \left(1-\tfrac{L_{\eta} \gamma_k}{2}\right)\gamma_k \|e_k\|^2.
\end{align*} 
					
					\noindent (ii) Summing \eqref{eqn:recursive_ineq_lemma5}  from $k =\ell, \ldots, K-1$ where $\ell\triangleq  \lceil \lambda K\rceil$ and taking expectations with respect to the iterate trajectory, we have
					\begin{align}
						\notag \sum_{{k=\ell}}^{K-1} \left( 1-\tfrac{ \sqrt{n} L_0\gamma_k}{\eta}\right) \tfrac{\gamma_k}{4}\EXP{\, \|G_{\eta,1/\gamma_k} (\x_k)\|^2\, }  
						& \leq   \EXP{\, f_{\eta}(\x_\ell)\, }-      \EXP{\, f_{\eta}(\x_{K})\, }\\ 
						&+\sum_{{k=\ell}}^{K-1}\left( 1-\tfrac{\sqrt{n} L_0\gamma_k}{2\eta}\right) \gamma_k \EXP{\, \|\mathbf{e}_k\|^2\, }.
					\end{align}
					Invoking the definition of $R_{\ell,K}$ and Lemma~\ref{lemma:mon_G}, we have 
					\begin{align} 
						\notag & \quad \left( \tfrac{1}{4}-\tfrac{\sqrt{n} L_0\gamma_0}{4\eta}\right) \left(\sum_{k=\ell}^{K-1} \gamma_k\right)\EXP{\mathbb{E}_{R_{\ell,K}}\left[\|G_{\eta,1/\gamma_0} ({\x_{R_{\ell,K}}})\|^2\right]}\\  
						& \notag \leq   \left( \tfrac{1}{4}-\tfrac{\sqrt{n} L_0\gamma_0}{4\eta}\right)\sum_{{k=\ell}}^{K-1} \gamma_k \mathbb{E}\left[\|G_{\eta,1/\gamma_0} ({\x_k})\|^2\right]   \leq						   \EXP{f_{\eta}(\x_\ell)-f^*_{\eta}}+\sum_{{k=\ell}}^{K-1} \gamma_k \EXP{\|\mathbf{e}_k\|^2}, 
					\end{align}
					where the last inequality follows from $\left( 1-\tfrac{\sqrt{n} L_0\gamma_k}{2\eta}\right) \leq 1$ for any $k$. 
					{Note that we have} 
					\begin{align}\notag
						\mathbb{E}\left[f_{\eta}(\x_{\ell})\right] - {f^*_{\eta}} & =   \mathbb{E}\left[ f(\x_{\ell}) + f_{\eta}(\x_{\ell})-f(\x_{\ell})\right]-f^*_{\eta} +f^*
						-f^*\\
						\notag & = \mathbb{E}\left[{f(\x_{\ell})}\right]-f^*+ \mathbb{E}\left[\left | f_{\eta}(\x_{\ell})-f(\x_{\ell})\right|\right] + \left| f^*-f^*_\eta\right| \\
						& \stackrel{\tiny{\hbox{Lemma~\ref{lemma:props_local_smoothing} (iii)}}}{\leq}
						\mathbb{E}\left[f(\x_{\ell})\right] -f^* +2L_0\eta.
						\label{bd_feta}
					\end{align}
					From Lemma \ref{lem:stoch_error_var} (ii), 
					$\mathbb{E}\left[\|\mathbf{e}_k\|^2\right] \leq \mathbb{E}\left[\mathbb{E}\left[\|\mathbf{e}_k\|^2\mid \mathcal{F}_k\right]\right] 
					\leq \tfrac{ 16\sqrt{2\pi}  L_0^2n}{N_k}.$
					From the preceding relations, we obtain the inequality \eqref{ineq:residual_map_ineq_VRGZO_2}.
				\end{proof}
				We now present an almost sure convergence guarantee for the sequence generated by Algorithm~\ref{algorithm:zo_nonconvex} by relying on the Robbins-Siegmund Lemma.
				\begin{proposition}[{\bf Asymptotic guarantees for \texttt{VRG-ZO}}]\label{prop:VRGZO_a.s.}\em
					Consider Algorithm \ref{algorithm:zo_nonconvex}. Let Assumptions~\ref{ass-1} and \ref{assum:random_vars} hold

					\noindent{\bf  (a)   [Constant stepsize rule]} Let $\gamma_k:= \gamma  < \tfrac{\eta}{\sqrt{n}L_0}$ and $N_k:=(k+1)^{1+\delta}$ for $k \geq 0$ and $\delta>0$. Then, the following hold. (a-i) $\|G_{\eta,1/\gamma}(\x_k)\| \xrightarrow[k \to \infty]{a.s.} 0.$ (a-ii) Every limit point of $\{\x_k\}$ lies in the set of $\eta$-Clarke stationary points of \eqref{eqn:prob} in an a.s. sense. 
					
					\noindent {\bf (b) [Diminishing stepsize rule]} Let $\{\gamma_k\}$ be a square summable but
					non-summable diminishing sequence with $\gamma_0 \in
					(0,\tfrac{\eta}{2\sqrt{n}L_0})$, and $\sum_{k} \tfrac{\gamma_k}{N_k} < \infty$.
					Then the following hold. (b-i) $\|G_{\eta,1/\gamma_0}(\x_k)\| \xrightarrow[k
					\to \infty]{a.s.} 0.$ (b-ii) Every limit point of $\{\x_k\}$ lies in the set
					of $\eta$-Clarke stationary points of \eqref{eqn:prob} in an a.s.  sense. 
				\end{proposition}
				\begin{proof}
					(a-i) Let $f^*_\eta\triangleq \displaystyle \inf_{\x \in \Xscr} f_\eta(\x)$. By taking expectations conditioned on $\mathcal{F}_k$ on the both sides of~\eqref{eqn:recursive_ineq_lemma5}, we have  
					\begin{align}
						\quad \mathbb{E}\left[(f_{\eta}(\x_{k+1}) -f^*_\eta) \mid \mathcal{F}_k\right]  
						& \leq      (f_{\eta}(\x_k)-f^*_\eta)- \left(1-L_{\eta}\gamma\right)\tfrac{\gamma}{4} \left\|{G}_{\eta,1/\gamma}(\x_k)\right\|^2\notag \\
						&+{\left( 1-\tfrac{\sqrt{n}L_0\gamma}{2\eta}\right)} \gamma \,  \mathbb{E}\left[{\|{\mathbf{e}_k}\|^2} \mid \mathcal{F}_k\right] \notag\\
						&\stackrel{\tiny{\hbox{Lemma \ref{lem:stoch_error_var}\hbox{ (ii)}}}}{\leq}  
						(f_{\eta}(\x_k)-f^*_\eta)- \left(1-L_{\eta}\gamma\right)\tfrac{\gamma}{4} \left\|{G}_{\eta,1/\gamma}(\x_k)\right\|^2\notag \\
						&+{\left( 1-\tfrac{{\sqrt{n}}L_0\gamma }{2\eta}\right)} \,\tfrac{{16 \gamma \sqrt{2\pi} L_0^2{n} }}{N_k},\label{cont-rec3}
					\end{align} 
					in view of Lemma \ref{lem:stoch_error_var} (ii). 
					By invoking Lemma~\ref{lemma:Robbins-Siegmund}, the nonnegativity of
					$f_{\eta}(\x_{k}) -f^*_\eta$ {and by recalling that $\gamma < \tfrac{\eta}{\sqrt{n}L_0}$}, we have that
					$\{(f_{\eta}(\x_k)-f^*_\eta)\}$ is convergent a.s. and
					$\sum_{k=1}^{\infty} \left\|{G}_{\eta,1/\gamma}(\x_k)\right\|^2 < \infty$
					almost surely. It remains to show that with probability one, $\|G_{\eta,1/\gamma}(\x_k)\| \to 0$ as $k \to \infty$.
					We proceed by contradiction. Suppose with finite probability,
					$\|G_{\eta,1/\gamma}(\x_k)\| \xrightarrow{k \in \tilde{\Kscr}} \tilde{\epsilon}
					> 0$ where $\tilde{\Kscr}$ is a random subsequence and $\tilde{\epsilon}$ is a random positive scalar. Consequently, there exists positive integer $\tilde{K}$ corresponding to random subsequence  $\tilde{\cal K}$ such that for all $k \in \tilde{\Kscr}$ such that $k \geq \tilde{K}$, $\| G_{\eta,1/\gamma}(\x_k) \| \geq
					\tfrac{\tilde{\epsilon}}{2}$.    Consequently, we have that $\sum_{k \to
						\infty} \|G_{\eta,1/\gamma}(\x_k)\|^2 \geq \sum_{k \in \tilde{\Kscr}}
					\|G_{\eta,1/\gamma} (\x_k)\|^2 \geq \sum_{k \in \tilde{\Kscr}, k \geq
						\tilde{K}} \|G_{\eta,1/\gamma} (\x_k)\|^2 = \infty$ with finite probability,
					leading to a contradiction and implying that $\|G_{\eta,1/\gamma}(\x_k)\|^2
					\xrightarrow[k \to \infty]{a.s.} 0.$ 
					
					\smallskip
					
					\noindent {(a-ii) Recall from Prop.~\ref{prop_equiv} that if $\x$ satisfies
						$G_{\eta,1/\gamma}(\x) = 0$, it is an $\eta$-Clarke stationary point of
						\eqref{eqn:prob}, {i.e.,} $0 \in \partial_{\lm{\eta}} f(\x)+\Nscr_{\Xscr}(\x).$
						Since almost every limit point of $\{\x_k\}$ satisfies
						$G_{\eta,{1}/\gamma}(\x) = 0$, the result follows.}

					\noindent (b-i)  Consider \eqref{eqn:recursive_ineq_lemma5},  whereby
					\begin{align}
						\quad \mathbb{E}\left[(f_{\eta}(\x_{k+1}) -f^*_\eta) \mid \mathcal{F}_k\right]  
						& \leq     (f_{\eta}(\x_k)-f^*_\eta)- \left(1-L_{\eta}\gamma_k\right)\tfrac{\gamma_k\left\|{G}_{\eta,1/\gamma_0}(\x_k)\right\|^2}{4} \notag \\
						\notag &+\left(1-\tfrac{L_{\eta} \gamma_k}{2}\right){\tfrac{\gamma_k{(16 \sqrt{2\pi} L_0^2{n} )}}{N_k}}. 
					\end{align} 
					{Based on the nonnegativity of}
					$f_{\eta}(\x_{k}) -f^*_\eta$, {$\gamma_k \leq \gamma_0 < \tfrac{\eta}{2{\sqrt{n}}L_0}$ for all $k$} and the summability of $\{\tfrac{\gamma_k}{N_k}\}$, {we may invoke} Lemma~\ref{lemma:Robbins-Siegmund}, implying that
					$\{(f_{\eta}(\x_k)-f^*_\eta)\}$ is convergent a.s. and
					$\sum_{k=1}^{\infty} \gamma_k \left\|{G}_{\eta,1/\gamma_0}(\x_k)\right\|^2 < \infty$
					a.s. . Since $\{\gamma_k\}$ is non-summable,  $\liminf_{k \to \infty} \left\|{G}_{\eta,1/\gamma_0}(\x_k)\right\|^2 = 0$ in an almost sure sense.  {{Proving} $\|G_{\eta,1/\gamma_0}(\x_k)\| \to 0$ as $k \to \infty$ follows in a similar fashion to as derived earlier.} (b-ii) {Recall that the residual mapping at a given vector is zero if and only if the vector is a stationary point. The result in (b-ii) follows by leveraging this property, the result in (b-i), and the continuity of the residual mapping.}   
				\end{proof}
				
				We now provide a formal rate (in terms of projection steps on $\Xscr$) and complexity statement (in terms of sampled function evaluations).
				
				\begin{theorem}[{\bf Rate and complexity statements for \texttt{VRG-ZO}}]\label{thm:vr-rb-zo}
					\em
					Consider Algorithm \ref{algorithm:zo_nonconvex}. Let Assumptions~\ref{ass-1} and \ref{assum:random_vars} hold.  Let $a > 0$ be an arbitrary scalar.
					
					{\noindent {\bf  (a)  [Constant step]}} Let $\gamma_k:=\gamma<{\frac{\eta}{\sqrt{n} L_0}}$ and $N_k:=  \lceil  a \sqrt{n}L_0 (k+1) \rceil$ for $k \geq 0$.  
					
					\noindent {\bf{(a-i)}} For $K> \tfrac{2}{1-\lambda}$ with $\ell\triangleq  \lceil \lambda K\rceil$, we have
					\begin{align}\label{eqn:GBound_constant_gamma}
						\mathbb{E}\left[ \|G_{\eta,1/\gamma} (\x_{R_{\ell,K}})\|^2\right] \leq 
						\tfrac{\left(\mathbb{E}\left[f(\x_{\ell})\right] -f^* +2L_0\eta+ \left( {16 \sqrt{n}} \sqrt{2\pi} L_0 \right)(0.5-\ln(\lambda)) {{a^{-1}}}\gamma\right)}{\left(\left( 1-\tfrac{{\sqrt{n}}L_0\gamma}{\eta}\right) \tfrac{\gamma}{4}(1-\lambda)K \right)}.
					\end{align}
					\noindent {\bf{(a-ii)}} Let $\gamma= {\tfrac{\eta}{2{\sqrt{n}}L_0}}$,  $\epsilon>0$, and $K_{\epsilon}$ satisfy $ \mathbb{E}\left[ \|G_{\eta,1/\gamma} (\x_{R_{\ell,K_{\epsilon}}})\|\right]   \leq \epsilon$.  Then the following hold.  {{(a-ii-1)}} The total no. of projection steps is $K_{\epsilon}=\mathcal{O}\left({n^{1/2}}{(L_0\eta^{-1} +L_0^2)} \epsilon^{-2}\right)$.
					
					\noindent {{(a-ii-2)}} The overall sample complexity is $\mathcal{O}\left({n^{3/2}(L_0^3\eta^{-2}+L_0^5)} \epsilon^{-4}\right)$.
					
					\noindent {\bf  (b)  [Diminishing step]} Let  $\gamma_k = \tfrac{\gamma_0}{\sqrt{k+1}}$ and {$N_k:=  {\lceil a{\sqrt{n}} L_0{\sqrt{k+1}} \rceil }$}, where {$\gamma_0<{{\eta}{{\sqrt{n}} L_0}}$}. 
				 {\bf{(b-i)}}  For $K> \tfrac{2}{1-\lambda}$ with $\ell\triangleq  \lceil \lambda K\rceil$, we have
					\begin{align}\label{eqn:GBound_diminish_gamma}
						\mathbb{E}\left[ \|G_{\eta,1/\gamma_0} (\x_{R_{\ell,K}})\|^2\right] \leq \tfrac{\left(\mathbb{E}\left[f(\x_{\ell})\right] -f^* +2L_0\eta+  \left({16  \sqrt{2\pi} L_0\sqrt{n}}\right)  (0.5-\ln(\lambda))\lm{a^{-1}}\gamma_0\right)}{\left(\left( 1-\tfrac{{\sqrt{n}} L_0\gamma_0}{\eta}\right) {\tfrac{\gamma_0}{2}}(1-\sqrt{\lambda})\sqrt{K}\right)}.
					\end{align}
					\noindent {\bf{(b-ii)}}  Let $\epsilon>0$ and $K_{\epsilon}$ be such that $ \mathbb{E}\left[ \|G_{\eta,1/{\gamma_0}} (\x_{R_{\ell,K_\epsilon}})\|\right]   \leq \epsilon$. Then the following hold. 
					 {{(b-ii-1)}} The total number of projection steps is  $K_{\epsilon}=\mathcal{O}\left({n(L_0^2\eta^{-2} +L_0^4)} \epsilon^{-4}\right).$ {{(b-ii-2)}} The total sample complexity is $\mathcal{O}\left({n^{2}(L_0^4\eta^{-3}+L_0^7)} \epsilon^{-6}\right)$.
				\end{theorem}
				\begin{proof}
					\noindent {\bf (a-i)} From~\eqref{ineq:residual_map_ineq_VRGZO_2}, we obtain
					\begin{align*}
						\mathbb{E}\left[ \|G_{\eta,{1/\gamma}} ({\x_{R_{\ell,K}}})\|^2\right]
						&\leq  
						\tfrac{\left( \mathbb{E}\left[f(\x_{\ell})\right] -f^* +2L_0\eta+ \left(16 \sqrt{2\pi} L_0^2{n} \right)\gamma\sum_{k=\ell}^{K-1} \tfrac{1}{N_k}\right)}{\left(\left( 1-\tfrac{\sqrt{n}L_0\gamma}{\eta}\right) \tfrac{\gamma}{4}(K-\ell)\right)} \\
						& \leq \tfrac{\left( \mathbb{E}\left[f(\x_{\ell})\right] -f^* +2L_0\eta+ \left({16 \sqrt{2\pi} L_0^2 n}\right){\gamma}\sum_{k=\ell}^{K-1} \tfrac{a^{-1}}{{L_0}{{\sqrt{n}}{(k+1)}}}\right)}{\left(\left( 1-\tfrac{\sqrt{n}L_0\gamma}{\eta}\right) \tfrac{\gamma}{4}(K-\ell)\right)},  
					\end{align*}
					where the expectation is with respect to both the iterate trajectory and the random integer $R_{\ell,K}$, while the last inequality  follows from the definition of $N_k$. Recall that $K> \tfrac{2}{1-\lambda}$ implies $\ell \leq K-1$ and  
					$ \textstyle\sum_{k={\ell}}^{K-1}\tfrac{1}{k+1}
					\leq 0.5 +\ln\left(\tfrac{N}{\lambda N+1}\right)\leq 0.5-\ln(\lambda).$
					{Further}, $K-\ell \geq K-\lambda K=(1-\lambda)K$ {implying that} \eqref{eqn:GBound_constant_gamma} holds.
					
					\noindent {\bf (a-ii)} To show (ii-1), using the relation in part (i) and by $\gamma:= \tfrac{\eta}{2\sqrt{n}L_0}$, we obtain
					\begin{align*}
						& \mathbb{E}\left[ \|G_{\eta,1/\gamma} (\x_{R_{\ell,K}})\|^2\right] \leq  
						\tfrac{\left(\mathbb{E}\left[f(\x_{\ell})\right] -f^* +2L_0\eta+ \left(8\eta \sqrt{2\pi} \right)  (0.5-\ln(\lambda)){a^{-1}}\right)}{\left(\eta(1-\lambda)K\right)/(16\sqrt{n}L_0)}.
					\end{align*}
					From Jensen's inequality,  $\mathbb{E}\left[ \|G_{\eta,1/\gamma} (\x_{R_{\ell,K}})\|\right] \leq \sqrt{\mathcal{O}\left({n^{1/2}}{(L_0\eta^{-1}+L_0^2)}/K\right)}$, implying that $K_{\epsilon}=\mathcal{O}\left({n^{1/2}}{(L_0\eta^{-1} +L_0^2)} \epsilon^{-2}\right)$. Next, we derive sample complexity as per (ii-2). 
					
					\begin{align*}
						\sum_{k=0}^{K_\epsilon} N_k & = \sum_{k=0}^{K_\epsilon}  \lceil  {a{\sqrt{n} L_0}(k+1)}\rceil 
						\leq {\sqrt{n} L_0}  {\mathcal{O}(K_{\epsilon}+{K_\epsilon^2})}
						\leq \mathcal{O}\left({n^{3/2}(L_0^3\eta^{-2}+L_0^5)} \epsilon^{-4}\right). 
					\end{align*}
					
					\noindent {\bf (b-i)}
					{Consider the inequality~\eqref{ineq:residual_map_ineq_VRGZO_2}.  Using $\tfrac{\gamma_k}{N_k}\leq \tfrac{\gamma_0}{{a{\sqrt{n}} L_0}(k+1)}$, we obtain} 
					\begin{align*}
						\mathbb{E}\left[ \|G_{\eta,{1/\gamma_0}} (\x_{R_{\ell,K}})\|^2\right]
						& \leq \tfrac{\left( \mathbb{E}\left[f(\x_{\ell})\right] -f^* +2L_0\eta+ \left({16  \sqrt{2\pi} L_0^2{n}} \right)\sum_{k=\ell}^{K-1} \tfrac{\gamma_0}{{a \sqrt{n} L_0}(k+1)}\right)}{\left(\left( 1-\tfrac{{\sqrt{n}} L_0\gamma_0}{\eta}\right) {\tfrac{\gamma_0}{2}}(1-\sqrt{\lambda})\sqrt{K}\right)},
					\end{align*}
					where we note that $\textstyle\sum_{k=\ell}^{K-1} \gamma_k \geq \textstyle\int_{\ell-1}^{K-1} \tfrac{\gamma_0}{\sqrt{x+1}} dx = 2\gamma_0(\sqrt{K}-\sqrt{\ell}) \geq 2\gamma_0(1-\sqrt{\lambda})\sqrt{K}.$ As earlier, we have that  
					$ \textstyle\sum_{k={\ell}}^{K-1}\tfrac{1}{k+1} \leq 0.5-\ln(\lambda),$ leading to the inequality in (b-i). 
					
					\noindent {\bf (b-ii)} From Jensen's
					inequality and by setting $\gamma_0 =
					\tfrac{\eta}{2 {\sqrt{n}} L_0}$, we obtain
					that \\$ \mathbb{E}\left[
					\|G_{\eta,1/\gamma_0}
					(\x_{R_{\ell,K}})\|\right] \leq
					\sqrt{\mathcal{O}\left({(L_0\eta^{-1}+L_0^2)\sqrt{n/K}}\right)}$, thus
					$K_{\epsilon}=\mathcal{O}\left(\tfrac{n(L_0^2\eta^{-2} +L_0^4)}{\epsilon^{4}}\right)$.\\ Next, (b-ii-2) holds by
					bounding the sample complexity as follows.
					\begin{align*}
						\sum_{k=0}^{K_\epsilon} N_k & = \sum_{k=0}^{K_\epsilon}  \lceil  {a\sqrt{n} L_0}\sqrt{k+1}\rceil 
						\leq {\sqrt{n} L_0}  {\mathcal{O}(K_{\epsilon}+{K_\epsilon^{3/2}})}
						\leq \mathcal{O}\left({n^{2}(L_0^4\eta^{-3}+L_0^7)} \epsilon^{-6}\right).
					\end{align*}
				\end{proof}
				\begin{remark}[{\bf Commentary on complexity and availability of $L_0$ in Theorem~\ref{thm:vr-rb-zo}}]\label{rem:VRG-ZO}
					\noindent First, we emphasize that the above rate and complexity guarantees correspond to a randomly chosen iterate, rather than the final iterate. We observe that \texttt{VRG-ZO} achieves better complexity bounds under a constant stepsize, i.e., $\gamma= {\tfrac{\eta}{2{\sqrt{n}}L_0}}$. However, from a practical perspective, the diminishing stepsize may seem appealing since the parameter $L_0$ is often unavailable, rendering a challenge in tuning the constant stepsize to establish convergence. In contrast, under a diminishing stepsize, even when $\gamma_0$ is chosen larger than the unknown threshold ${\tfrac{\eta}{{\sqrt{n}}L_0}}$, the stepsize $\gamma_k$ will eventually fall below the threshold for which the error bound \eqref{eqn:GBound_diminish_gamma} is satisfied.
				\end{remark}

\section{A Smoothed Quasi-Newton Framework}\label{sec:SQN}
In this section, we present a stochastic quasi-Newton framework for nonsmooth
nonconvex stochastic optimization. After introducing a smoothed
unconstrained formulation in~\cref{sec:sqn_formulation}, we present the SQN
algorithm in~\cref{sec:qn_description} and develop a damped L-BFGS scheme
in~\cref{sec:two_loop}. Guarantees of convergence are provided
in~\cref{sec:qn_conv}. 
\subsection{A smoothed unconstrained formulation}\label{sec:sqn_formulation}
Consider the unconstrained reformulation of \eqref{eqn:prob}, defined as 
\begin{align}\label{eqn:unconstrained}
    \min_{\x}  \quad h(\x), \mbox{ where } h(\x) \, \triangleq \, f(\x) + {\bf 1}_{\Xscr}(\x)\mbox{ where }  f(\x) \, \triangleq \,  \mathbb{E}\left[\, \f(\x,\bxi)\, \right],
\end{align}
and
$\mathbf{1}_{\Xscr}(\x)$ denotes the indicator function of the set $\Xscr$. By leveraging  randomized smoothing of $f$ and Moreau smoothing of $\mathbf{1}_{\Xscr}$, we define the smoothed counterpart of $h$ as  
\begin{align}\label{eqn:double_smoothed}
    \min_{\x} \quad  h_{\eta}{(\x)} \, &\triangleq \, f_{\eta}(\x)  + \mathbf{1}_{\Xscr,\eta}(\x), 
\end{align}
where  $f_\eta (\x)\, \triangleq \, \mathbb{E}_{\uu \in \mathbb{B}}\left[\, f(\x+\eta \uu)\, \right],
\mathbf{1}_{\Xscr, \eta}(\x)\, \triangleq \, {\displaystyle  \min_{y \in \mathbb{R}^n}} \left\{ \mathbf{1}_{\Xscr}(y) + \frac{1}{2 \eta} \| \x - \y \|^2 \right\}$,
 $\mathbb{B}\triangleq \{u \in \mathbb{R}^n\mid \|u\|\leq 1\}$, and $\eta>0$ denotes a smoothing parameter.
We observe that $\mathbf{1}_ {\Xscr, \eta} (\x)$ is the Moreau-smoothed indicator function of $\Xscr$. Crucially, $\mathbf{1}_{\Xscr, \eta} \in C^{1,1}$ with a $(1/\eta)$-Lipschitz continuous gradient defined as
$\nabla \mathbf{1}_{\Xscr, \eta}(\x) =  \tfrac{1}{\eta}\left(\x - \Pi_{\Xscr}(\x)\right).$ By leveraging Lemma~\ref{lemma:props_local_smoothing}, for a  $\eta$, $h_\eta \in C^{1,1}$ with $L_\eta$-Lipschitz continuous gradient where $L_{\eta} \triangleq \tfrac{L_{0}{\sqrt{n}}+1}{\eta}$. 
\begin{algorithm}[htb]
     \caption{\texttt{VRSQN-ZO}: A VR zeroth-order smoothed quasi-Newton method}\label{algorithm:quasi-newton}
    \begin{algorithmic}[1]\fyy{
         \STATE\textbf{input:} $\x_0 \in \Xscr$,  $\{\gamma_k\} >0$, $\mu > 0$,  $\{N_k\}$,  $K$, and memory parameter $\tilde{p} \geq1$. {Define $p := \min \{k, \tilde{p}\}$.}
    \FOR {$k=0,1,\ldots,{K}-1$}
     \FOR {$j = 1, \ldots, N_k$}  
         \STATE {Generate a random sample $v_{j,k} \in {\eta} \mathbb{S}$ and  
        compute the stochastic ZO gradient}  
\begin{align*}
&g_{\eta}(\x_k,v_{j,k},\xi_{j,k}) :=\tfrac{n\left(\f(\x_k+ v_{j,k},\xi_{j,k}) - \f(\x_k-v_{j,k},\xi_{j,k})\right)v_{j,k}}{2\left\|v_{j,k}\right\|{\eta}} + \tfrac{\x_k - \Pi_{\Xscr}(\x_k)}{\eta}.
\end{align*}
     \ENDFOR
      
	\STATE  { Evaluate ${\bar{g}_{\eta,{N_k}}(\x_k)} :=  {\tfrac{\sum_{j=1}^{N_k}  g_{{\eta}}(\x_k,v_{j,k},\xi_{j,k})}{N_k}}$.}  
   	
%
    
    \STATE Discard the vectors   $ s_{k-p},y_{k-p},\bar{y}_{k-p}$ from storage
    
%
\STATE  Generate $r_k$ by passing ${\bar{g}_{\eta,{N_k}}(\x_k)}$, $s_{k-p}, \dots s_{k-1}$, $ {y}_{k-p}, \dots,  {y}_{k-1}$, and $\bar{y}_{k-p}, \dots, \bar{y}_{k-1}$ to Algorithm~\ref{algorithm:hessian-construction} (\texttt{L-BFGS-SMOOTHED}).
             \STATE Update $\x_k$ as follows. $\x_{k+1}:=\x_k-\g_k r_k$.  \hspace{.3in} { $\vartriangleright$ {\bf L-BFGS update }}
                  
    \FOR {$j = 1, \ldots, N_{k}$}  
     \STATE Use the generated samples $v_{j,k} \in {\eta} \mathbb{S}$  to compute
\begin{align*}
&g_{{\eta}}(\x_{k+1},v_{j,k},{\xi}_{j,k}) :=\tfrac{n\left({\f}(\x_{k+1}+ v_{j,k},{\xi}_{j,k}) - {\f}(\x_{k+1}-v_{j,k},{\xi}_{j,k})\right)v_{j,k}}{{2\left\|v_{j,k}\right\|{\eta}}} + \tfrac{\x_{k+1} - \Pi_{\Xscr}(\x_{k+1})}{{\eta}} . 
\end{align*}

     \ENDFOR
    	\STATE  {Evaluate  ${\widehat{g}}_{{\eta}, {N_k}}(\x_{k+1})  := \tfrac{\sum_{j=1}^{N_k}  g_{{\eta}}(\x_{k+1},{v_{j,{k}}},{\xi}_{j,{k}})}{N_k}$.}  
        	
    \STATE {Evaluate $s_{k}:=\x_{k+1} - \x_{k}$ and $y_{k}:= {\widehat{g}}_{\eta,N_{k}}(\x_{k+1})-\bar{g}_{\eta,N_{k}}(\x_{k}) $.}

     \STATE {Evaluate $\bar{y}_{k}   := \Phi_{k}y_{k} + (1 - \Phi_{k})H^{-1}_{k+1,0}s_{k}$ where $\Phi_k$ is defined by \eqref{theta-choice}.}
 
    \ENDFOR

        \STATE Return $\x_K$} 
\end{algorithmic}
\end{algorithm}
Next, we present the assumptions and comment on prior work.
\begin{assumption}[{\bf Problem properties}] 
	\label{ass-2}\em
	\noindent Consider \fyy{problem}~\eqref{eqn:double_smoothed}.
	
    \noindent (i) Function~$\tilde{f}(\bullet,{\bxi})$ is $L_0$-Lipschitz continuous on $\mathbb{R}^n$ {a.s. for all $\bxi$}.  
 (ii)  $\mathcal{X}$ is nonempty, closed, and convex.
 (iii) For all $\x \in \mathbb{R}^n$,  $\mathbb{E}[\, {\f(\x,\bxi)} \, \mid \, \x\,] \, = \, f(\x)$ a.s. . 

	
\end{assumption}
\begin{remark} (i) Most prior work on SQN
schemes~\cite{goldfarb,Wang2019StochasticPQ} requires twice differentiability of
$\tilde{f}(\bullet,{\bxi})$, such that $\|\nabla^2_{\x \x}\tilde{f}(\x,{\bxi}) \|
\leq \kappa$ for all $\x,{\bxi}$ and some $\kappa>0$ (e.g., see AS.5
in~\cite{goldfarb}). We considerably weaken this assumption by allowing the
objective function to be Lipschitz continuous and possibly nondifferentiable.
(ii) In the prior section, $\tilde{f}(\bullet,{\bxi})$ was required to be
$L_0$-Lipschitz on $\Xscr + \eta_0 \mathbb{B}$, whereas we now require that it
is $L_0$-Lipschitz on $\mathbb{R}^n.$ We believe this assumption can be weakened by
considering problem of the form \eqref{eqn:unconstrained} where the objective
function is augmented by a smooth nonconvex function, with possibly unbounded
gradients over $\mathbb{R}^n$. The study of such extensions is left as a future
direction to our work. 
\begin{table}[h]
	\scriptsize
	\centering
		\caption{Comparison of assumptions and eigenvalue bounds across relevant works.}
	\label{comp_QN}
	\begin{tabular}{@{}l >{\raggedright\arraybackslash}p{1.5cm} >{\raggedright\arraybackslash}p{1.5cm} c >{\raggedright\arraybackslash}p{2.5cm}@{}}
		\toprule
		\multicolumn{5}{c}{\textbf{Assumptions in Relevant Works}} \\
		\midrule
		\textbf{Authors} & $f \in C^1$ & $f \in C^2$ & $\nabla f$ is Lip. cont. &  $\exists \kappa > 0$ such that $\|\nabla^2_{xx} f(x, \xi)\| \leq \kappa$ \\ 
		\hline
		Wang, Ma, Goldfarb, Liu~\cite{goldfarb} & \cmark & \cmark & \cmark & \cmark \\
		\hline
		Bollapragada and Wild~\cite{bollapragada2023adaptive} & \cmark & \xmark & \cmark & \xmark \\
		\hline
		Berahas, Nocedal, Tak\'{a}\v{c}~\cite{berahas2016} & \cmark & \cmark & \cmark & \cmark  \\
		\hline
		\rowcolor{pink}
		$f(\bullet)$ (this work) & \xmark & \xmark & \xmark & \xmark \\
		\hline
		\rowcolor{pink}
		$f_{\eta}(\bullet)$ (this work) & \cmark & \xmark & \cmark & \xmark \\
		\bottomrule
	\end{tabular}
	
	\vspace{1.5em}

	\begin{tabular}{@{}l cc@{}}
		\toprule
		\multicolumn{3}{c}{\textbf{Eigenvalue Bounds}} \\
		\midrule
		\textbf{Authors} & $\underline{\lambda}$ & $\overline{\lambda}$ \\ 
		\midrule
		Wang, Ma, Goldfarb, Liu~\cite{goldfarb} & $\left( \frac{4p \kappa^2}{\delta} + (4p + 1)(\kappa + \delta) \right)^{-1}$ & $\left( \frac{\left(\frac{4\kappa + 5\delta}{\delta}\right)^{2p} - 1}{\left(\frac{4\kappa + 5\delta}{\delta}\right)^2 - 1} \right) \frac{4}{\delta} + \frac{\left(\frac{4\kappa + 5\delta}{\delta}\right)^{2p}}{\delta}$ \\
		\hline
		Bollapragada and Wild~\cite{bollapragada2023adaptive} & $C_1$ & $C_2$ \\ 
		\hline
		Berahas, Nocedal, Tak\'{a}\v{c}~\cite{berahas2016} & $C_1$ & $C_2$ \\
		\hline \rowcolor{pink}
		This work & $\tfrac{\delta}{32(2+\delta)(p+1)L_\eta^2}$ & $\left(4p + 1 \right)\left( 1 + \tfrac{16 L_\eta\sqrt{2+\delta}}{\delta} \right)^{2p}$ \\
		\bottomrule
	\end{tabular}

\end{table}
 \end{remark}

\subsection{Algorithm Description}\label{sec:qn_description} 
We now present a zeroth-order SQN algorithm, called \texttt{VRSQN-ZO} and
presented by Algorithm~\ref{algorithm:quasi-newton}, for addressing
problem~\eqref{eqn:unconstrained}. Some of the key characteristics of this
method are as follows. (i) \texttt{VRSQN-ZO} is primarily a derivative-free
quasi-Newton method, reliant on the {sampled} objective function
evaluations. (ii) It leverages randomized smoothing of the objective function
and Moreau smoothing of the constraint set. (iii) It employs a damped
limited-memory BFGS scheme called \texttt{L-BFGS-SMOOTHED} (see
Algorithm~\ref{algorithm:hessian-construction}) for efficiently constructing
inverse Hessian approximations of the expectation-valued nonconvex objective
function. (iv) Lastly, our method is equipped with variance reduction which
improves the iteration complexity in stochastic regimes.

Next, we describe the outline of \texttt{VRSQN-ZO}. The sequence
$\{\x_k\}$ denotes the iterates generated by the method. Similar to
\texttt{VRG-ZO}, at iteration $k$ in \texttt{VRSQN-ZO}, a mini-batch of random
samples $v_{j,k} \in \eta\mathbb{S}$ of size $N_k$ is generated and used for
evaluating of $\bar{g}_{\eta,N_k}(\x_k) \in \mathbb{R}^n$, a
derivative-free estimate (with respect to $f$) of $\nabla
h_{\eta}(\x_{k})$. A key difference with \texttt{VRSQN-ZO} lies in the
use of an inverse Hessian approximation in updating $\x_k$. As shown 
later, the main update rule of \texttt{VRSQN-ZO} can be compactly represented
as $\x_{k+1} = \x_{k} - \gamma_k H_k\bar{g}_{\eta,N_k}(\x_k) $ where
$H_k$ denotes a zeroth-order inverse Hessian approximation of the smoothed
function. This framework is inspired by {but distinct from}~\cite{goldfarb}; it is
(i) {\em derivative-free} and (ii) {\em can accommodate nonsmoothness in the nonconvex term}. In
computing the product $H_k \bar{g}_{\eta,N_k}(\x_k)$, we employ a
stochastic damped L-BFGS scheme. This scheme employs two auxiliary stochastic
zeroth-order gradients evaluated at successive iterates. Let us define
\begin{align} \label{gbar}
& {\widehat{g}}_{{\eta}, {N_k}}(\x_{k+1})   \triangleq \tfrac{\sum_{j=1}^{N_k}  g_{{\eta}}(\x_{k+1},{\bv_{j,{k}}},\xi_{j,{k}})}{N_k}, \quad 
  {\bar{g}_{\eta,{N_k}}(\x_k)}   \triangleq   {\tfrac{\sum_{j=1}^{N_k}  g_{{\eta}}(\x_k,\bv_{j,k},\xi_{j,k})}{N_k}},\\
   & \mbox{where }
\label{g-def}
    g_{{{\eta}}}(\bullet,\bv_{j,k},\xi_{j,k}) \triangleq \tfrac{n\left({\f}(\bullet+ \bv_{j,k},\xi_{j,k}) - {\f}(\bullet - {\bv_{j,k}},\xi_{j,k})\right)\bv_{j,k}}{{2}\left\|\bv_{j,k}\right\|{\eta}} + \tfrac{1}{{\eta}}(\bullet - \Pi_{\Xscr}(\bullet)).
\end{align}
{As shown in Algorithm~\ref{algorithm:quasi-newton}, ${\widehat{g}}_{{\eta}, {N_k}}(\x_{k+1})$ is evaluated after $\x_{k+1}$ is computed during the iteration $k$. However, ${\bar{g}}_{{\eta}, {N_k}}(\x_{k})$ is used in the computation of $\x_{k+1}$ and so, it is evaluated prior to the main L-BFGS step. Note that although ${\widehat{g}}_{{\eta}, {N_k}}(\x_{k+1})$ is computed at iteration $k$, it is used at time $k+1$. Importantly, this circumvents the need for the storage of the generated mini-batch for the next iteration and helps with memory efficiency. For $k \geq 0$,  let} the stochastic gradient difference \fyy{$y_{k}$ and iterate difference $s_{k}$} be defined as
\begin{align}\label{eqn:y_s_definition}
  y_{k}\triangleq  {\widehat{g}}_{\eta,N_{k}}(\x_{k+1})-\bar{g}_{\eta,N_{k}}(\x_{k})   
    \mbox{ {and} }  
    s_{k}\triangleq \x_{k+1} - \x_{k}, \mbox{respectively.}
\end{align}
We now define $\bar{y}_{k}$ with a view towards satisfying the curvature condition {$s_{k}^\top \bar{y}_{k} > 0$.   
\begin{align}
  \label{bary}  \bar{y}_{k}  &= \Phi_{k}y_{k} + (1 - \Phi_{k})H^{-1}_{k+1,0}s_{k}\\    
  \label{theta-choice}  \mbox{ where }\Phi_{k} & = \begin{cases}  \tfrac{0.75 s^{\fyy{\top}}_{k} H^{-1}_{k+1,0}s_{k}}{s^{\fyy{\top}}_{k} H^{-1}_{k,0}s_{k}-s^{\fyy{\top}}_{k}y_{k}}, \quad & \mbox{if   } s^{\fyy{\top}}_{k}y_{k} < 0.25 s^{\fyy{\top}}_{k} H^{-1}_{k+1,0}s_{k}\\
 1 \, , & \mbox{otherwise} \, 
\end{cases}    
\end{align}
and given a user-defined constant $\delta > 0 $, for $k\geq 1$, $H_{k,0}$ is defined as 
\begin{align}\label{bound-on-first-hessian}
    H_{k,0} = \nu^{-1}_{k} \mathbf{I},  \mbox{ where }  \nu_{k} = \max\left\{ \tfrac{y^{\fyy{\top}}_{k-1}y_{k-1}}{s^{\top}_{k-1}y_{k-1}+ \delta s_{k-1}^\top s_{k-1}}, \delta \right\}.
\end{align}
Observe that the additional term $\delta s_{k-1}^{\top} s_{k-1}$ allows us to
develop suitable boundedness claims.  In addition, this implies that
$\|H_{k,0}\| = \tfrac{1}{\nu_k}.$  In contrast with~\cite{goldfarb}, we
provide bounds on the smallest and largest eigenvalue of $H_k$ in terms of
$\eta$.  Next, we study the properties of the first and second moments of
random errors in the zeroth-order approximation of the gradient map. {We observe that a partial overlap-style scheme such as the one presented in \cite{berahas2016} could be considered, but proceed with a full overlap scheme for reasons of simplicity.}
Throughout, we define the history of the algorithm as earlier in
\eqref{eqn:history} by $\mathcal{F}_k$ at iteration $k$. 
\begin{definition}\label{def:stoch_errors_sqn} 
    For all $k\geq 0,j=1,\ldots,N_k$, let $\tilde{\mathbf{e}}_k\triangleq \tfrac{\sum_{j=1}^{N_k}\tilde{\mathbf{e}}_{j,k}}{N_k}$,
\begin{align}
    {\tilde{\mathbf{e}}_{j,k}} & \triangleq \left({g_{\eta}(\x_k,\vv_{j,k},\bxi_{j,k})}+\tfrac{1}{\eta}\left(\x_k - \Pi_{\Xscr}(\x_k)\right)\right)-\nabla h_\eta(\x_k),\label{def:stoch_errors1_sqn}
\end{align}
    where realizations of $g_{\eta}(\x_k,\vv,\bxi)$ and $\tilde{\mathbf{e}}_k$ are  given by \eqref{g-def} and  $\tilde{e}_k$. $\hfill \Box$
\end{definition}
The following result can be shown in a similar vein to the proof of Lemma~\ref{lem:stoch_error_var}.
\begin{lemma}[{\bf Bias and moment properties of {$\tilde{\mathbf{e}}$}}]\label{lem:stoch_error_var_sqn}\em
Let Assumptions \ref{ass-2} and \ref{assum:random_vars} hold. Then (i) and (ii) hold almost surely for $k\geq 0$ and $N_k \geq 1$.

    \noindent (i) $\mathbb{E}[\tilde{\mathbf{e}}_{j,k}\mid \mathcal{F}_k] =0$ for all $j=1,\ldots,N_k$. (ii) $\mathbb{E}[\|\tilde{\mathbf{e}}_k\|^2\mid \mathcal{F}_k] \leq  \tfrac{{16 \sqrt{2\pi }n L_{0}^{2}}}{N_k}$. $\hfill \Box$ 
\end{lemma} 

\subsection{Construction of the Inverse Hessian Approximation}\label{sec:two_loop}
{Let us define the sequence of limited-memory BFGS matrices $\{B_k\}$ for $k \geq p$ as follows. 
\begin{align}\label{LBFGS-update}
    B_{k,i} & = B_{k,i-1} + \tfrac{\bar{y}_{j}\bar{y}_{j}^\top}{s_{j}^\top\bar{y}_{j}} - \tfrac{B_{k,i-1}s_{j}s_{j}^\top B_{k,i-1}}{s_{j}^\top B_{k,i-1}s_{j}}, \mbox{where }j = k-p+i-1, \  i = 1,\hdots,p \\
    \label{Bk_update}
     {B_{k}} &\triangleq  B_{k,p}.  
\end{align}
Also, for $k\geq p$, we define the inverse Hessian approximation matrix $H_k$ as follows. 
\begin{align}\label{Hk_update}
H_{k,i} & =V_j^\top H_{k,i-1}V_j + \rho_j s_{j}s^{\fyy{\top}}_{j} ,  \mbox{ where }     j = k-p+i-1, \  i = 1,\hdots,p,  \\
H_{k} &\triangleq  H_{k,p},   \notag
\end{align}
where $
 \rho_j\triangleq \frac{1}{\bar{y}_{j}^{\fyy{\top}}s_{j}}, \hbox{ and }V_j\triangleq \mathbf{I}-\rho_{j}\bar{y}_{j}s_{j}^{\fyy{\top}}$, for all  $j=k-p,\ldots, k-1.$ Using the Sherman-Morrison-Woodbury formula, we have $H_{k} \triangleq B_{k}^{-1}$ for all $k \geq p$.
Our proposed  ZO SQN scheme uses a damped limited-memory BFGS update
rule. The outline of this scheme is provided in
Algorithm~\ref{algorithm:hessian-construction}. Here, given a memory parameter
$p\geq 1$, the vectors $\{s_j,y_j,\bar{y}_j\}$ for $i=k-p,\ldots,k-1$, and
$\bar{g}_{\fyy{\eta},{N_k}}(\x_k)$, the two-loop scheme
\texttt{L-BFGS-SMOOTHED} generates the product
$H_k{\bar{g}_{\eta,{N_k}}(\x_k)}$, without the need to compute or store matrix
$H_k$ when $k\geq 1$.
 \begin{algorithm}[hbt]
     \caption{\texttt{L-BFGS-SMOOTHED}: Zeroth-Order Damped Limited-Memory BFGS}\label{algorithm:hessian-construction}
    \begin{algorithmic}[1]{
         \STATE\textbf{input:}  $\bar{g}_{\eta,N_k}(\x_k)$, iterate differences $s_{k-p}, \dots, s_{k-1}$ and gradient estimate differences $y_{k-p}, \dots, y_{k-1}$ and $\bar{y}_{k-p}, \dots, \bar{y}_{k-1}$, and a user-defined constant $\delta>0$.
\STATE Initialize $H_{k,0}:= \nu_{k}^{-1}{\bf I}$ where $\nu_{k} =  \max\left\{ \tfrac{y^{\top}_{k-1}y_{k-1}}{s^{\top}_{k-1}y_{k-1}+  {\delta s_{k-1}^{\top} s_{k-1}}}, \delta \right\}.$ \hspace{.1in} \tikzmark{right} 
 
     \STATE Initialize $q_0:=\bar{g}_{\eta,N_k}(\x_k)$ 
     \FOR {$j=k-1:k-p$}
      \STATE Compute $\rho_j := (s_{j}^\top \bar{y}_{j})^{-1}$
     \STATE Compute scalar $\alpha_{k-j}:=\rho_j s_{j}^\top q_{k-j-1} $ 
     \STATE Update vector $q_{k-j}:=q_{k-j-1}-\alpha_{k-j}\bar{y}_{j}$ 
     \ENDFOR
     \STATE Initialize vector $r_0:=H_{k,0}q$ 
     \FOR {$j=k-p:k-1$}
     \STATE  Update vector $r_{j-k+p+1}:=r_{j-k+p}+\left(\alpha_{k-j}-\rho_j \bar{y}_{j}^\top r_{j-k+p} \right)s_{j}$ 
          \ENDFOR
        \STATE Return  $r_p$  
        }
\end{algorithmic}
\end{algorithm}
Next, we show that the generated iterate by the proposed ZO algorithm is indeed a well-defined stochastic damped L-BFGS scheme.
\begin{proposition}\label{lemma:Bk_posdef}\em
Let $\{\x_k\}$ be generated by Algorithm~\ref{algorithm:quasi-newton}, i.e.,
$$\x_{k+1} := \x_k -\gamma_k d_k,  \qquad \hbox{where } d_k \triangleq 
\begin{cases}
               \bar{g}_{\eta,N_{k}}(\x_{k}) , \quad \hbox{if } k <p,\\
               r_k,\ \ \  \quad\qquad \hbox{if } k \geq p,\\
\end{cases}$$
where $r_k$ is returned by Algorithm~\ref{algorithm:hessian-construction} at iteration $k$. Then, the following results hold.  

\noindent (a) $r_k=B_k^{-1}{\bar{g}_{\eta,{N_k}}(\x_k)}$ for $k\geq p$ where the sequence $\{B_k\}$ is defined by \eqref{LBFGS-update}-\eqref{Bk_update}.

\noindent (b) The pair of $s^{\fyy{\top}}_{k}$ and $\bar{y}_{k}$ satisfies the curvature conditon. More precisely, we have \begin{align*}
s^{\top}_{k} \bar{y}_{k} \geq 0.25 s^{\top}_{k} H^{-1}_{k,0}s_{k}, \qquad \hbox{for all } k\geq 0. 
\end{align*}  

\noindent (c) For any $k \geq p$, if $B_{k,0} \succ 0$, then $B_{k,i}$ and $H_{k,i}$ are positive definite for all $i   =  1 ,\ldots, p$, respectively.

    \noindent (d)  For all $k\geq p$, $\mathbb{E}\left[\, \mathbf{H}_k\,\mid\, \mathcal{F}_k\, \right] =\mathbf{H}_k$ and  $\mathbb{E}\left[\, \mathbf{r}_k \, \mid \, \mathcal{F}_k\,\right] =\mathbf{H}_k\nabla h_\eta(\x_k)$ almost surely, where $h_\eta$ is given by \eqref{eqn:unconstrained}. 

\end{proposition}
\begin{proof} \noindent  \fyy{(a) Consider \eqref{Hk_update}. By unrolling this update rule recursively, we have
\begin{align}\label{H_k_nonrec}
H_{k,p} &=\left(\textstyle\prod_{i=1}^{p}V_{j}\right)^{\top}H_{k,0}\left(\textstyle\prod_{i=1}^{p}V_{j}\right)  + \rho_{k-p}\left(\textstyle\prod_{i=2}^{p}V_{j}\right)^{\top}s_{k-p}s^{\fyy{\top}}_{k-p}\left(\textstyle\prod_{i=2}^{p}V_{j}\right) \notag\\
& + \rho_{k-p+1}\left(\textstyle\prod_{i=3}^{p}V_{j}\right)^{\top}s_{k-p+1}s^{\fyy{\top}}_{k-p+1}\left(\textstyle\prod_{i=3}^{p}V_{j}\right)  + \ldots\notag\\
& +\rho_{k-2}V^{\top}_{k-1}s_{k-2}s_{k-2}^{\top}V_{k-1} +\rho_{k-1}s_{k-1}s_{k-1}^{\top},   
\end{align}
where $ j = k-p+i-1$. Consider Algorithm~\ref{algorithm:hessian-construction}. Note that $q_0\triangleq \bar{g}_{\eta,N_{k}}(\x_{k})$ and $r_0\triangleq H_{k,0}q_{p}$. Next, we derive a formula for $q$. We have
\begin{align*}
q_{k-j}&=q_{k-j-1} -\alpha_{k-j}\bar{y}_j = q_{k-j-1} - \rho_j\left(s_j^{\top}q_{k-j-1}\right)\bar{y}_j = q_{k-j-1} - \rho_j\left(y_js_j^{\top}\right)q_{k-j-1} \\
&= \left( \mathbf{I}-\rho_{j}y_js_j^{\fyy{\top}}\right)q_{k-j-1}=V_jq_{k-j-1}, \quad \hbox{for all } j=   k-1, \ldots, k-p.
\end{align*}
From the preceding relation, we may write
\begin{align}\label{q_k_nonrec}
q_{\ell}=\left(\textstyle\prod_{i=p-\ell+1}^{p}V_{k-p+i-1}\right)q_0, \quad \hbox{for all } \ell= 1,2, \ldots, p. 
\end{align}
Further, from the algorithm, we have $\alpha_1=\rho_{k-1}s_{k-1}^{\top}q_0$ and 
\begin{align}\label{alpha_k_nonrec}
\alpha_{\ell}=\rho_{k-\ell}s_{k-\ell}^{\top}\left(\textstyle\prod_{i=p-\ell+2}^{p}V_{k-p+i-1}\right)q_0, \quad \hbox{for all } \ell= 2,3, \ldots, p. 
\end{align}
Multiplying both sides of \eqref{H_k_nonrec} by $q_0$ and employing \eqref{q_k_nonrec} and \eqref{alpha_k_nonrec}, we obtain
\begin{align}\label{H_kq_nonrec}
H_{k,p}q_0 &=\left(\textstyle\prod_{i=1}^{p}V_{k-p+i-1}\right)^{\fyy{\top}}H_{k,0}q_p + \left(\textstyle\prod_{i=2}^{p}V_{k-p+i-1}\right)^{\fyy{\top}}s_{k-p}\alpha_p \\
& + \left(\textstyle\prod_{i=3}^{p}V_{k-p+i-1}\right)^{\fyy{\top}}s_{k-p-1}\alpha_{p-1}  +\ldots+V^{\fyy{\top}}_{k-1}s_{k-2}\alpha_2+s_{k-1}\alpha_1.\notag
\end{align}
Next, we derive a formula for $r$. From the algorithm, we have 
\begin{align*}
r_{j-k+p+1}&=r_{j-k+p}+\left(\alpha_{k-j}-\rho_jy_j^{\top}r_{j-k+p}\right)s_j \\
&= r_{j-k+p}-\rho_js_j y_j^{\top}r_{j-k+p}+\alpha_{k-j}s_j\\
& = V_j^{\fyy{\top}}r_{j-k+p}+\alpha_{k-j}s_j, \quad \hbox{for all } j= k-p,\ldots,k-1.
\end{align*}
This implies that $
r_{\ell}= V_{k-p+\ell-1}^{\top}r_{\ell-1}+\alpha_{p-\ell+1}s_{k-p+\ell-1}$, for all $\ell= 1,2,\ldots,p,$
where recall that $r_0\triangleq H_{k,0}q_{p}$. Unrolling the preceding relation recursively, we obtain 
\begin{align}
r_p &=\left(\textstyle\prod_{i=1}^{p}V_{k-p+i-1}\right)^{\top}r_0 + \alpha_p\left(\textstyle\prod_{i=2}^{p}V_{k-p+i-1}\right)^{\top}s_{k-p} \\
& +\alpha_{p-1} \left(\textstyle\prod_{i=3}^{p}V_{k-p+i-1}\right)^{\top}s_{k-p-1}  +\ldots+\alpha_2V^{\top}_{k-1}s_{k-2}+\alpha_1s_{k-1}.\notag
\end{align}
From the preceding relation and \eqref{H_kq_nonrec}, we may obtain the result by noting that $$r_p=H_{k,p}q_0 = H_k\bar{g}_{\eta,N_{k}}(\x_{k}) = B_k^{-1}\bar{g}_{\eta,N_{k}}(\x_{k}). $$


\noindent (b) This follows from the definition of $\Phi_k$ and the proof {can be done similar to Lemma~3.1 in~\cite{goldfarb}}.

\noindent (c)  This can be shown} in a fashion similar to~\cite[Lemma 3.1]{goldfarb}. \fyy{To elaborate, from part (b), we have} 
$ s^{\top}_{k-1} \bar{y}_{k-1} \geq 0.25 s_{k-1}^{\top}H^{-1}_{k,0}s_{k-1}	$. For any nonzero vector $a \in \mathbb{R}^n$,   
\begin{align*}
a^{\top} H_{k,i} a & = a^{\top} \left( \left(I - \tfrac{s_{j}\bar{y}^{\top}_{j}}{s_{j}^{\top}\bar{y}_{j}}\right)H_{k,i-1}\left(I -\tfrac{s_{j}\bar{y}^{\top}_{j}}{s_{j}^{\top}\bar{y}_{j}}\right) + \tfrac{s_{j}s^{\top}_{j}}{s_{j}^{\top}\bar{y}_{j}} \right) a  \\
& = a^{\top}\left(I - \tfrac{s_{j}\bar{y}^{\fyy{\top}}_{j}}{s_{j}^{\fyy{\top}}\bar{y}_{j}}\right)H_{k,i-1}\left(I -\tfrac{s_{j}\bar{y}^{\fyy{\top}}_{j}}{s_{j}^{\fyy{\top}}\bar{y}_{j}}\right) a + a^{\top}\tfrac{s_{j}s^{\fyy{\top}}_{j}}{s_{j}^{\fyy{\top}}\bar{y}_{j}} a > 0 \, ,
\end{align*}
given that $H_{k,i-1} \succ 0$. Thus, choosing an initial $H_{k,0} \succ 0$ ensures $H_{k,i} \succ 0$ for all $i$. 

\noindent (d) First note that the update rule \eqref{Hk_update} implies that $H_k$ is computed using the terms $\{s_j,y_j\}$ for $i=k-p,\ldots,k-1$. Recall that from \eqref{eqn:y_s_definition} we have 
\begin{align*} 
y_{k-1}\triangleq  {\widehat{g}}_{\eta,N_{k-1}}(\x_{k})-\bar{g}_{\eta,N_{k-1}}(\x_{k-1})   
   \quad \mbox{ and } \quad 
s_{k-1}\triangleq \x_{k} - \x_{k-1}.
\end{align*}
In view of the definition of the history of the method in \eqref{eqn:history}, this implies that $\mathbf{H}_k$ is $\mathcal{F}_k$-measurable and so, we have $\mathbb{E}\left[\, \mathbf{H}_k\, \mid \, \mathcal{F}_k\, \right] =\mathbf{H}_k$. Thus, from part (a), we may write 
$$\mathbb{E}\left[\, \mathbf{r}_k \, \mid \, \mathcal{F}_k\, \right]
=\mathbf{H}_k \mathbb{E}\left[\, \bar{g}_{\eta,N_{k}}(\x_{k}) \, \mid \, \mathcal{F}_k \, \right]
=\mathbf{H}_k \mathbb{E}\left[\, \nabla h_\eta(\x_k) + \tilde{\mathbf{e}}_k \,\mid \, \mathcal{F}_k\, \right]
=\mathbf{H}_k \nabla h_\eta(\x_k), $$
where the preceding relation is implied by Lemma~\ref{lem:stoch_error_var_sqn}.
\end{proof}

\subsection{Convergence Analysis}\label{sec:qn_conv}
We begin by proving an intermediate result that provides a bound on $\tfrac{\left\|y_{k-1}\right\|^2}{\left\|s_{k-1}\right\|^2}$ based on $L_{\eta}$, where $L_{\eta} \triangleq \tfrac{L_{0}{\sqrt{n}}+1}{\eta}.$
\begin{lemma} \label{yj-bound} \em
Let Assumption~\ref{ass-2} hold. Then $\tfrac{\left\|y_{k-1}\right\|^2}{\left\|s_{k-1}\right\|^2} \leq \fyy{4L_\eta^2}$ for all $k \geq 1$.
\end{lemma}
\begin{proof} 
From the definitions of $y_{k-1}$, $\hat{g}_{{\fyy{\eta}},\fyy{N_{k-1}}}(\x_k)$, {and}  $\bar{g}_{{\fyy{\eta}},\fyy{N_{k-1}}}(\x_{k-1})$,
\begin{align*}
    & y_{k-1}  =   \tfrac{1}{N_{k-1}}  \left( \sum_{j=1}^{N_{k-1}} \left( \tfrac{n\left({\f}(\x_k+ v_{j,k-1},\xi_{j,k-1}) - {\f}(\x_k {- v_{j,k-1}},\xi_{j,k-1})\right)v_{j,k-1}}{{2}\left\|v_{j,k-1}\right\|\eta} + \tfrac{\left(\x_k - \Pi_{\Xscr} (\x_k) \right)}{\eta}\right) \right.  \\
    & - \left.   \sum_{j=1}^{N_{k-1}} \left(\tfrac{n\left({\f}(\x_{k-1}+ v_{j,k-1},\xi_{j,k-1}) - \f(\x_{k-1} {- v_{j,k-1}},\xi_{j,k-1})\right)v_{j,k-1}}{{2}\left\|v_{j,k-1}\right\|\eta} + \tfrac{\left(\x_{k-1} - \Pi_{\Xscr}(\x_{k-1})\right)}{\eta} \right)\right). 
\end{align*}
    Taking norms on both sides and invoking the triangle inequality, non-expansivity of the Euclidean projector, and $L_0$-Lipschitz continuity, we obtain
    \begin{align*}
        & \|y_{k-1}\| \\
	& \hspace{-0.2in}= \tfrac{1}{N_{k-1}}\left\| \sum_{j=1}^{N_{k-1}}\left(\tfrac{n\left({\f}(\x_{k-1}{- v_{j,k-1}},\xi_{j,k-1}) - \f(\x_k{- v_{j,k-1}},\xi_{j,k-1})\right)v_{j,k-1}}{{2}\left\|v_{j,k-1}\right\|\eta}\right) - \tfrac{N_{k-1}\left(\Pi_{\Xscr} (\x_{k}) -\Pi_{\Xscr} (\x_{k-1})\right)}{\eta} \right.\\
        &+  \left. \sum_{j=1}^{N_{k-1}} \left( \tfrac{n\left(\f(\x_k+ v_{j,k-1},\xi_{j,k-1}) - \f(\x_{k-1}+ v_{j,k-1},\xi_{j,k-1})\right)v_{j,k-1}}{{2}\left\|v_{j,k-1}\right\|\eta}\right) + \tfrac{N_{k-1}\left(\x_k - \x_{k-1} \right) }{{\fyy{\eta}}} \right\|\\  
    &\leq \tfrac{1}{N_{k-1}}\left( \sum_{j=1}^{N_{k-1}} \tfrac{n L_{0}\left\|s_{k-1}\right\|}{{\eta}} + \tfrac{N_{k-1}}{{\eta}}\left\|s_{k-1}\right\|
    +  \sum_{j=1}^{N_{k-1}} \tfrac{n L_{0}\left\|s_{k-1}\right\|}{\eta} + \tfrac{N_{k-1}}{\eta}\left\|s_{k-1}\right\|  \right) \\
        & =  \left( \tfrac{2n L_0}{ \eta  + \tfrac{2}{\eta}}\right)\left\|s_{k-1}\right\| = \fyy{2 L_\eta\left\|s_{k-1}\right\|}.
\end{align*}
Squaring both sides of the above inequality, we obtain the result.
\end{proof}
We now present bounds on the eigenvalues for each iterate in the sequence $\left\{\, H_k\, \right\}$.
\smallskip

\begin{proposition}\label{prop:lambdak_bounds}\em
Suppose $H_k$ is constructed as in {\eqref{Hk_update}}, with $H_0 = \delta I$. Suppose $\delta \eta^2 \leq 4$. Let the smallest and largest eigenvalue of $H_{k}$ be denoted by  $\underline{\lambda}_k$ and $\overline{\lambda}_k$, respectively. Then for all $k$,  
\begin{align}\label{eqn:eigenvalue_bounds}
 \underline{\lambda}_k & \geq  \underline{\lambda}_{\eta,\delta,p} \triangleq    \tfrac{ \delta }{32(2+\delta)(p+1)L_\eta^2} \mbox{ and } \\ 
 \overline{\lambda}_k   &   \leq  \overline{\lambda}_{\eta,\delta,p} \triangleq  \left(4p + 1 \right)\left( 1 + \tfrac{16 L_\eta\sqrt{2+\delta}}{\delta} \right)^{2p}.  
\label{eqn:eigenvalue_bounds2}
\end{align}
\end{proposition}
\begin{proof}
\noindent {(i)}
    We obtain a lower bound on the lowest eigenvalue of $H_k$ by finding an upper bound on the largest eigenvalue of $B_k$. Under \fyy{the L-BFGS} update rule for $B_k$, 
\begin{align*}
B_{k,i} = B_{k,i-1} + \tfrac{\bar{y}_{j}\bar{y}_{j}^{\fyy{\top}}}{s_{j}^{\fyy{\top}}\bar{y}_{j}} - \tfrac{B_{k,i-1}s_{j}s_{j}^{\fyy{\top}}B_{k,i-1}}{s_{j}^{\fyy{\top}} B_{k,i-1}s_{j}}, 
\end{align*}
where $k$ is the iteration index, $i = 1,\hdots, p$, and $j = k-p+i-1.$
Since $B_{k,i-1} \succ 0$, $ s_{j}^{\fyy{\top}} B_{k,i-1}s_{j} > 0.$ For convenience, \fyy{let us define $a_{j}^2\triangleq s_{j}^{\fyy{\top}} B_{k,i-1}s_{j}$ and} $w_{j} \triangleq B_{k,i-1}s_{j}$. Rearranging terms to apply the triangle inequality, and using the notation we have introduced, we may write the above as
\begin{align}\label{pre-triangle}
B_{k,i} = B_{k,i-1}  - \tfrac{w_{j} w_{j}^{\fyy{\top}}}{a_{j}^2} + \tfrac{\bar{y}_{j}\bar{y}_{j}^{\fyy{\top}}}{s_{j}^{\fyy{\top}}\bar{y}_{j}} \, .
\end{align}
Considering the first two terms on the right\fyy{-}hand side, we observe that the matrix norm of the difference may be bounded as
\begin{align} 
\left\|B_{k,i-1} - \tfrac{w_{j} w_{j}^{\fyy{\top}}}{a_{j}^2} \right\| &= \max_{ \|u\|=1} \left(u^{\fyy{\top}}\left(B_{k,i-1} - \tfrac{w_{j} w_{j}^{\fyy{\top}}}{a_{j}^2}\right)u \right) \nonumber 
 =  \max_{ \|u\|=1} \left( u^{\fyy{\top}} B_{k,i-1}u - \tfrac{u^{\fyy{\top}}w_{j} w_{j}^{\fyy{\top}}u}{a_{j}^2} \right)  \\
& = \max_{ \|u\|=1} \left( u^{\fyy{\top}} B_{k,i-1}u \right) - \left\|\tfrac{w_{j}^{\fyy{\top}}u}{a_{j}} \right\|^2 n
 \leq  \|B_{k,i-1}\| \, . \label{matrix-bound}
\end{align}
Finally applying the triangle inequality to \eqref{pre-triangle} combined with \eqref{matrix-bound}
\begin{align} 
\|B_{k,i}\| & \leq \left\|B_{k,i-1} - \tfrac{B_{k,i-1}s_{j}s_{j}^{\fyy{\top}}B_{k,i-1}}{s_{j}^{\fyy{\top}} B_{k,i-1}s_{j}} \right\| + \left\|\tfrac{\bar{y}_{j}\bar{y}_{j}^{\fyy{\top}}}{s_{j}^{\fyy{\top}}\bar{y_{j}}}\right\| \leq \|B_{k,i-1}\| + \tfrac{\bar{y}_{j}^{\fyy{\top}}\bar{y}_{j}}{s_{j}^{\fyy{\top}}\bar{y_{j}}}. \label{ineq-bnd}
\end{align}
Note from \eqref{theta-choice} that $\bar{\Phi}_{k-1}$ is chosen so that 
\begin{align*}
    \tfrac{\bar{y}_{j}^{\fyy{\top}}\bar{y}_{j}}{s_{j}^{\fyy{\top}}\bar{y}_{j}} & \leq \tfrac{4 \left\|\Phi_j y_j + (1-\Phi_j)s_{j}^{\fyy{\top}} B_{j+1,0}s_j \right\|^2}{s_{j}^{\fyy{\top}} B_{j+1,0}s_j} 
  \leq 8 \Phi_j^2 \tfrac{y_{j}^{\fyy{\top}}y_{j}}{s_{j}^{\fyy{\top}} B_{j+1,0}s_j} + 8 \left( 1- \Phi_j\right)^2 \tfrac{\left\|s_j^{\fyy{\top}} B_{j+1,0}s_j \right\|^2}{s_{j}^{\fyy{\top}} B_{j+1,0}s_j}  \\
   & \hspace{-0.2in} \leq 8 \Phi_j^2 \tfrac{y_{j}^{\fyy{\top}}y_{j}}{s_{j}^{\fyy{\top}} B_{j+1,0}s_j} + 8 \left( 1- \Phi_j\right)^2 \tfrac{ \| B_{j+1,0}\|^2 \|s_j\|^2 }{s_{j}^{\fyy{\top}} B_{j+1,0}s_j}
   \leq 8 \Phi_j^2 \tfrac{y_{j}^{\fyy{\top}}y_{j}}{s_{j}^{\fyy{\top}} B_{j+1,0}s_j} + 8 \left( 1- \Phi_j\right)^2 \tfrac{ \nu_{j+1}^2 \|s_j\|^2 }{\nu_{j+1} \|s_{j}\|^2} \\
    &  \hspace{-0.2in} \leq 8 \Phi_j^2 \tfrac{y_{j}^{\fyy{\top}}y_{j}}{s_{j}^{\fyy{\top}} B_{j+1,0}s_j} + 8 \left( 1- \Phi_j\right)^2 v_{j+1} =  8 \Phi_j^2 \tfrac{y_{j}^{\fyy{\top}}y_{j}}{\nu_{j+1}s_{j}^{\fyy{\top}} s_{j}} +  8 \left( 1- \Phi_j\right)^2 \nu_{j+1}.
\end{align*}
    Invoking Lemma~\ref{yj-bound} and noting that $\Phi_j \leq 1$,
    \begin{align}
    \tfrac{\bar{y}_{j}^{\fyy{\top}}\bar{y}_{j}}{s_{j}^{\fyy{\top}}\bar{y}_{j}} & \leq  \tfrac{8 \Phi_j^2}{\nu_{j+1}} \fyy{\left(4L_\eta^2\right)}   + 8 (1-\Phi_j)^2\nu_{j+1}  \overset{\tiny (\nu_{j+1} \, \geq\, \delta)}{\leq} \tfrac{8}{{\delta}}  \fyy{\left(4L_\eta^2\right)}  + 8\left(\tfrac{y_j^{\fyy{\top}}y_j}{s_j^{\fyy{\top}}y_j + \delta s_j^{\fyy{\top}}s_j} + \delta\right) \notag\\ 
  \notag  & \overset{(y_j^{\fyy{\top}}s_j \, >\, 0)}{\leq} \tfrac{8}{{\delta}}  \fyy{\left(4L_\eta^2\right)}  + 8  \left( \tfrac{\|y_j\|^2}{\delta \|s_j\|^2}+\delta\right)   \overset{ (\fyy{\delta \eta^2 \leq 4})}{\leq} \tfrac{8}{{\delta}}  \fyy{\left(4L_\eta^2\right)}  + 8  \left( \tfrac{\|y_j\|^2}{\delta \|s_j\|^2}+\fyy{\left(4L_\eta^2\right)} \right) \\ 
        \label{bd_yj_sj} & {\leq} \,\tfrac{8}{{\delta}} \,\fyy{\left(4L_\eta^2\right)}  + \left(\tfrac{8+8\delta}{\delta}\right)\fyy{\left(4L_\eta^2\right)}  
     = \left(\tfrac{16+8\delta}{\delta}\right) \fyy{\left(4L_\eta^2\right)} .
\end{align}
Combining this with \eqref{ineq-bnd} yields
\begin{align}
 & \notag  \left[  \|B_{k,i}\|  \leq \|B_{k,i-1}\| + \left( \tfrac{16+8\delta}{{\delta}}\right) \fyy{\left(4L_\eta^2\right)}  \right] \\
& \notag\overset{\tiny \mbox{Inductively}}{\implies}    
    \left[ \|B_{k,i}\|  \leq \|B_{k,0}\| +  \left( \tfrac{16+8\delta}{\delta}\right) \sum_{i=1}^p \fyy{\left(4L_\eta^2\right)} \right] \\
& \implies   \left[ \|B_{k,i}\|  \leq \|B_{k,0}\| +  \left( \tfrac{16+8\delta}{\delta}\right) p
\fyy{\left(4L_\eta^2\right)} \right].
\end{align}
    Since $\|B_{k,0} \|= \nu_k \overset{\eqref{bd_yj_sj}}{\leq} \tfrac{(8+8\delta)}{\delta}\fyy{\left(4L_\eta^2\right)} $,
\begin{align}
    \|B_k\| & \leq \left(\tfrac{8+8\delta}{\delta}\right)\left(4L_\eta^2\right)  + \tfrac{(16+8\delta)p}{\delta}  \left(4L_\eta^2\right)   \leq \left(\tfrac{(16+8\delta)(p+1)}{\delta}\right) \left(4L_\eta^2\right).
\end{align}
This implies that for all $k$,  $\underline{\lambda}_k   \geq \tfrac{ \delta }{32(2+\delta)(p+1)L_\eta^2} .$  

\noindent {(ii)}
	Under the stochastic damped L-BFGS update rule \eqref{Hk_update}, where $k\geq p$, $i = 1,\hdots,p$, $j = k-p+i-1,$ \fyy{ we have}
	\begin{align*}
		H_{k,i} & = H_{k,i-1} - \rho_{j}(H_{k,i-1}\bar{y}_{j}s^{\fyy{\top}}_{j} +s_{j}\bar{y}_{j}^{\fyy{\top}} H_{k,i-1}) + \rho_{j}s_{j}s^{\fyy{\top}}_{j} + \rho_{j}^2(\bar{y}^{\fyy{\top}}_{j}\fyy{H_{k,i-1}}\bar{y}_{j})s_{j}s^{\fyy{\top}}_{j}.
	\end{align*}
	{By noting that} $\rho_{j} s_{j}^{\fyy{\top}} s_{j} = \tfrac{s_{j}^{\fyy{\top}} s_{j}}{s_{j}^{\fyy{\top}}\bar{y}_j}$,  $\tfrac{\|\bar{y}_{j}\| \|\bar{s}_{j}\|}{s_j^{\fyy{\top}} \bar{y}_{j}} = \left( \tfrac{\|\bar{y}_{j}\|^2 \|\bar{s}_{j}\|^2}{(s_j^{\fyy{\top}} \bar{y}_{j})(s_j^{\fyy{\top}} \bar{y}_{j})} \right)^{1/2}$, \\ and $\bar{y}^{\fyy{\top}}_{j}\fyy{H_{k,i-1}}\bar{y}_{j}\leq \|\fyy{H_{k,i-1}}\| \|\bar{y}_j\|^2$, {this allows for deriving the following bound.}
	\begin{align*}
		\|H_{k,i}\| & \leq \|H_{k,i-1}\| + \tfrac{2 \|H_{k,i-1}\| \|\bar{y}_{j}\| \|s_{j}\| + s^{\fyy{\top}}_{j}s_{j}}{s^{\fyy{\top}}_{j}\bar{y}_{j}} + \tfrac{s_{j}^{\fyy{\top}}s_{j} \|\fyy{H_{k,i-1}}\| \|\bar{y}_{j}\|^2 }{\left(s^{\fyy{\top}}_{j}\bar{y}_{j}\right)^2} \\
		& = \|H_{k,i-1}\| + 2 \|H_{k,i-1}\| \left( \tfrac{\|\bar{y}_{j}\|^2 \|\bar{s}_{j}\|^2}{(s_j^{\fyy{\top}} \bar{y}_{j})(s_j^{\fyy{\top}} \bar{y}_{j})} \right)^{1/2} + \|H_{k,i-1}\| \tfrac{\|s_{j}\|^2  \|\bar{y}_{j}\|^2 }{(s^{\fyy{\top}}_{j}\bar{y}_{j})^2} + \tfrac{\|s_{j}\|^2}{s^{\fyy{\top}}_{j}\bar{y}_{j}}   \, . 
	\end{align*}
	Recall from \eqref{bd_yj_sj} that
	\begin{align}\label{bd_sj_ybar_j_1}
		\tfrac{\|\bar{y}_{j}\|^2}{s_j^{\fyy{\top}} \bar{y}_{j}} \leq  \left(\tfrac{16+8\delta}{{\delta}}\right)\fyy{\left(4L_\eta^2\right)} .    
	\end{align}
	{In addition, \fyy{in view of Proposition~\ref{lemma:Bk_posdef},} we have the following bound $\tfrac{s_{j}^{\fyy{\top}} s_{j}}{s_{j}^{\fyy{\top}}\bar{y}_j}$.
		\begin{align}\label{bd_sj_ybar_j_2}
			\tfrac{s_{j}^{\fyy{\top}} s_{j}}{s_{j}^{\fyy{\top}}\bar{y}_j}  \leq \tfrac{\|s_j\|^2}{0.25 s_j^{\fyy{\top}} \fyy{H_{j,0}^{-1}} s_j} \leq \tfrac{\|s_j\|^2}{0.25 \nu_j \|s_j\|^2} \leq \tfrac{4}{\delta}.  
	\end{align}}
	{By invoking \eqref{bd_sj_ybar_j_1} and \eqref{bd_sj_ybar_j_2}, we obtain that
		\begin{align} \notag
			\|H_{k,i}\| 
			& \leq  \left(1 + 2\left(\tfrac{4(16+8\delta)}{{\delta^2}}  \fyy{\left(4L_\eta^2\right)}\right)^{1/2}  + \tfrac{4(16+8\delta)}{{\delta^2}} \fyy{\left(4L_\eta^2\right)} \right)  \|H_{k,i-1}\| + \tfrac{4}{\delta}\\
			& \leq  \underbrace{\fyy{\left(1 +   \left(\tfrac{8 \sqrt{2+\delta}}{{\delta}}\right)  \fyy{\left(2L_\eta\right)}\right)^2}}_{\, \triangleq \, \fyy{\phi}} \|H_{k,i-1}\| + \tfrac{4}{\delta}.
			\label{phi-bound}
	\end{align}}
		Consequently, the above inequality can be rewritten as 
		$\|H_{k,i}\|  \leq  \fyy{\phi} \|H_{k,i-1}\| + \tfrac{4}{\delta}. $ Noting that $\|H_{k,0}\| \leq \tfrac{1}{\delta}$, and proceeding recursively by letting $i$ run from $1,\hdots,p$, and observing that $\fyy{\phi} \geq 1$, \fyy{we may write}
		$\|H_{k,i}\|   \leq  \tfrac{\fyy{\phi}^{p}}{\delta} + \tfrac{4}{\delta}\sum_{l=0}^{p-1} \fyy{\phi}^{l}
		\leq  \left(4p + 1 \right)\tfrac{{\fyy{\phi}}^{p}}{\delta}.$  Consequently, \fyy{we obtain} {$\overline{\lambda}_k \leq \left(4p + 1 \right)\left( 1 + \tfrac{\fyy{16L_\eta}\sqrt{2+\delta}}{\delta} \ \right)^{2p}.$} 
\end{proof}
	\fyy{Next, we provide a recursive bound on the smoothed function and establish the asymptotic convergence in terms of the smoothed function in an almost sure sense.}
\begin{proposition}[{\bf Asymptotic guarantees for \texttt{VRSQN-ZO}}]\label{prop:sqn_as}\em
\fyy{Let $\{\x_k\}$ be generated by Algorithm~\ref{algorithm:quasi-newton} and let Assumptions~\ref{ass-2} and \ref{assum:random_vars} hold. Suppose $\gamma_k \leq \frac{\underline{\lambda}_{\eta,\delta,p}}{\overline{\lambda}_{\eta,\delta,p}^2L_\eta}$ for all $k$ where $\underline{\lambda}_{\eta,\delta,p}$ and $\overline{\lambda}_{\eta,\delta,p}$ are given by \eqref{eqn:eigenvalue_bounds}--\eqref{eqn:eigenvalue_bounds2}. Let $\fyy{\delta \eta^2 \leq 4}$. Then, the following holds for all $k\geq 0$.}

\noindent (i) $  \ \mathbb{E}\left[ h_{\eta}(\x_{k+1}) \, \mid \, \mathcal{F}_k\right] \leq h_{\eta}(\x_k)  -\left(\tfrac{\fyy{\underline{\lambda}_{\eta,\delta,p}}}{2}\right)  \gamma_k \|\nabla h_{\eta}(\x_k) \|^2 + \left({{8} n\sqrt{2\pi}L_0^2 L_{\eta} } \overline{\lambda}^2_{\eta,\delta,p}\right)\tfrac{\gamma_{k}^2}{N_k} .$
\noindent (ii) \fyy{Further, suppose $\{\gamma_k\}$ and $\{N_k\}$ satisfy  $\sum_{k=0}^{\infty}\gamma_k = + \infty$ and $\sum_{k=0}^{\infty}\frac{\gamma_{k}^{2}}{N_k} < \infty$. Then, the} sequence $\{\left\| \nabla_{\x} h_{\eta}(\x_k)\right\|^2\}$ converges to zero a.s. as $k \to \infty$.
	\end{proposition}
	\begin{proof}
	\fyy{(i)} By invoking the $L_{\eta}$-smoothness of $h_{\eta}$ and the boundedness of largest eigenvalue $\overline{\lambda}_k$ of $H_k$, we obtain
		\begin{align*}
			h_{\eta}(\x_{k+1})  &\, \leq \, h_{\eta}(\x_k)   + \left(\nabla h_{\eta}(\x_k)\right)^{\top}(\x_{k+1}-\x_{k}) + \tfrac{L_{\eta}}{2}\|\x_{k+1}-\x_{k}\|^2\\
			& \, = \, h_{\eta}(\x_k) - \gamma_k \left( \nabla h_{\eta}(\x_k)\right)^{\top} H_k \fyy{\bar{g}_{\eta,N_k}(\x_k)}  + \tfrac{L_{\eta}}{2}\gamma_{k}^2\|H_k \fyy{\bar{g}_{\eta,\fyy{N_k}}(\x_k)}\|^2\\
			& \, \leq \,  h_{\eta}(\x_k)  - \gamma_k \left( \nabla h_{\eta}(\x_k)\right)^{\top} H_k \nabla h_{\eta}(\x_k)   \\
			& \, - \,  \gamma_{k}\left( \nabla h_{\eta}(\x_k)\right)^{\top} H_k {\tilde{e}_k} + \tfrac{L_{\eta}}{2}\gamma_{k}^2\overline{\lambda}_k^2\|\bar{g}_{\eta,N_k}(\x_k)\|^2,
		\end{align*}
		where $ \bar{g}_{\eta,N_k}(\x_k)\,   =\,   \nabla h_{\eta}(\x_k) +{\tilde{e}_k}$ in view of Definition~\ref{def:stoch_errors_sqn}. Taking conditional expectations with respect to $\mathcal{F}_{k}$,
        \begin{align}\label{eqn:robbins_siegmund}
			\mathbb{E}\left[h_{\eta}(\x_{k+1}) \mid \mathcal{F}_{k}  \right] & \leq h_{\eta}(\x_k) - \gamma_k \left( \nabla h_{\eta}(\x_k)\right)^{\top} H_k \nabla h_{\eta}(\x_k) \notag \\
            & + \tfrac{L_{\eta}}{2}\gamma_{k}^2\overline{\lambda}^{2}_{k}\mathbb{E}\left[ \| \nabla h_{\eta}(\x_k) +{\tilde{\mathbf{e}}_k} \|^2  \mid \mathcal{F}_{k} \right],
        \end{align}
		where the last inequality is implied by Proposition~\ref{lemma:Bk_posdef}. We note that for any $k$,
       \begin{align*}
            \mathbb{E}\left[ \| { \nabla h_{\eta}(\x_k) +{\tilde{\mathbf{e}}_k}} \|^2  \mid \mathcal{F}_{k} \right]  &= \mathbb{E}\left[ \|\nabla h_{\eta}(\x_k) \|^2  \mid \mathcal{F}_{k} \right] 
             +  \mathbb{E}\left[ \| {\tilde{\mathbf{e}}_k}  \|^2  \mid \mathcal{F}_{k} \right] \\ 
             + 2 \underbrace{\mathbb{E}\left[ {\tilde{\mathbf{e}}_k}  ^{\top} \nabla h_{\eta}(\x_k)   \mid \mathcal{F}_{k} \right]}_{ {\, = \, 0}}  		& \leq \|\nabla h_{\eta}(\x_k) \|^2 + {\tfrac{16 n \sqrt{2\pi} L_0^2}{N_k}},
        \end{align*}
		where the last relation is implied by invoking Lemma~\ref{lem:stoch_error_var_sqn}.
		Combining \fyy{the preceding bound} with \eqref{eqn:robbins_siegmund}, we obtain {the following for $k \geq 0$.} 
		\begin{align*} 
			\mathbb{E}\left[h_{\eta}(\x_{k+1})  \mid{\cal F}_{k}  \right]   \leq h_{\eta}(\x_k)  -\left(\gamma_k \underline{\lambda}_k - \tfrac{L_{\eta}\overline{\lambda}_k^2 \gamma_k^2}{2} \right) \|\nabla h_{\eta}(\x_k)\|^2  + \tfrac{{8}n\sqrt{{2}\pi}{L_0^2}L_{\eta} \overline{\lambda}_k^2\gamma_{k}^2}{N_{k}} ,
		\end{align*}
		{where we use} $\left( \nabla h_{\eta}(\x_k)\right)^{\top} H_k \nabla h_{\eta}(\x_k) \geq \underline{\lambda}_k \|\nabla h_{\eta}(\x_k)\|^{\fyy{2}}$. \fyy{The bound in (i) follows by invoking the bounds provided in Proposition~\ref{prop:lambdak_bounds} and then recalling that $\gamma_k \leq \frac{\underline{\lambda}_{\eta,\delta,p}}{\overline{\lambda}_{\eta,\delta,p}^2L_\eta}$. }

\noindent {(ii) Let $h_{\eta}^{*}\triangleq \min_{\x } h_\eta(\x)$ where $h_\eta(\x)$ is given by \eqref{eqn:double_smoothed}. From part (i), we have 
$$ \mathbb{E}\left[ h_{\eta}(\x_{k+1})-h_{\eta}^{*} \, \mid \, \mathcal{F}_k\right] \leq h_{\eta}(\x_k)-h_{\eta}^{*}  - \tfrac{\underline{\lambda}_{\eta,\delta,p}\gamma_k \|{\nabla h_{\eta}}(\x_k) \|^2}{2}   +  \tfrac{{{8} n}\sqrt{2\pi}{ L_0^2}L_{\eta}\overline{\lambda}^2_{\eta,\delta,p}\gamma_{k}^2}{N_k}   .$$
} We proceed by 
		invoking \fyy{Lemma~\ref{lemma:Robbins-Siegmund}}. From {the non-summability of \fyy{$\gamma_k$}}, the summability of
		\fyy{$\frac{\gamma_{k}^2}{N_k}$}, and the nonnegativity of
		$h_{\eta}(\x_{k}) -h^*$, we have that
		$\{(h_{\eta}(\x_k)-h^*)\}$ is convergent a.s. and
		$\sum_{k=1}^{\infty} \left\|\nabla h_{\eta}(\x_k)\right\|^2 < \infty$
		almost surely. It remains to show that with probability one, $\|\nabla h_{\eta}(\x_k)\|^2\to 0$ as $k \to \infty$. This can be shown by contradiction, in a similar vein to proof of Proposition~\ref{prop:VRGZO_a.s.}.
		\end{proof}
Since we employ a smoothed counterpart of the indicator function of ${\cal X}$, the sequence of iterates are not guaranteed to be feasible.  We now proceed to derive a bound on the infeasibility of $\x^{*,\eta}$.
			\begin{lemma}\em
	Let $f$ be $L_0$-Lipschitz and $h_{\eta}{(\x)} \, \triangleq \, f_{\eta}(\x)  + \mathbf{1}_{\Xscr,\eta}(\x)$. 

\noindent (i) Let $\nabla h_{\eta}(\x^{*,\eta}) = 0$. Then the infeasibility of $\x^{*,\eta}$ is at most {$ 4 (2\pi)^{1/4}n^{1/2} L_0 {\eta} $}.

\noindent {(ii) Let $\|\nabla h_{\eta}(\x)\| \le \epsilon$. Then the infeasibility of $\x$ is at most $\epsilon + 4 (2\pi)^{1/4}n^{1/2} L_0 {\eta}$.} 
\end{lemma}	
\begin{proof}
	Let $\x^{*,\eta}$ be a stationary point of the smoothed problem, implying that $\nabla h_{\eta} (\x^{*,\eta}) = 0.$ Since $\nabla h_{\eta}(\x^{*,\eta}) = \nabla_{\x} f_{\eta}(\x^{*,\eta}) + \nabla_{\x} {\bf 1}_{\Xscr,\eta}({\x^{*,\eta}}) = 0$, 
	\begin{align*}
	- \tfrac{1}{\eta} {\left(\x^{*,\eta} - \Pi_{\Xscr}\left[ \x^{*,\eta} \right]\right)} = 	   \nabla_{\x}f_{\eta}(\x^{*,\eta}) = \left(\tfrac{n}{2\eta}\right) \mathbb{E}_{\vv \in \eta \mathbb{S}} \left[{({{f}}(\x^{*,\eta}+\vv) - {{f}}(\x^{*,\eta}-\vv)) \tfrac{\vv}{\|\vv\|}}\right]. 
	\end{align*}
	Invoking Lemma~\ref{lemma:props_local_smoothing} part (v),
	\begin{align*}
				\left\| \, \nabla_{\x}f_{\eta}(\x^{*,\eta}) \, \right\|^2  & = \left\| \left(\tfrac{n}{{2}\eta}\right) \mathbb{E}_{\vv \in \eta \mathbb{S}} \left[{\left({f}(\x^{*,\eta}+\vv) - {f}(\x^{*,\eta}-\vv)\right) \tfrac{\vv}{\|\vv\|}}  \right] \, \right\|^2 {\le}  16\sqrt{2\pi}L_0^2 n. 
	\end{align*}
	Consequently, $\| \x^{*,\eta} - \Pi_{\Xscr} \left[ \x^{*,\eta} \right] \| \leq 4 (2\pi)^{1/4} L_0 \sqrt{n}{\eta}$. 
\noindent (ii) {This follows from the definition of $\nabla h_{\eta}(\x)$ and the reverse triangle inequality{, written as
$$\tfrac{1}{\eta} \| { \x - \Pi_{\Xscr}\left[ \x \right] }\| - \|\nabla_{\x} f_{\eta}(\x) \| \leq \|\nabla_{\x} f_{\eta}(\x) +\tfrac{1}{\eta} {\left(\x - \Pi_{\Xscr}\left[ \x \right]\right)} \|.$$}}
\end{proof}
\begin{remark} (i) We observe that the {a.s.} convergence holds when, for example, $\gamma_k$ is square-summable but non-summable while $N_k \ge 1$. In fact, convergence also follows if $\gamma_k = \gamma$ for every $k$ where $\gamma$ is sufficiently small, but $\sum_{k=1}\tfrac{1}{N_k} < \infty$.   (ii) In addition, a limit point of $\left\{ \x_k \right\}$ may not be feasible with respect to \eqref{eqn:prob}. For this reason, the stationarity with respect to the smoothed problem \eqref{eqn:double_smoothed} does not ensure $\lm{\eta}$-stationarity with respect to the original problem \eqref{eqn:prob}. However, if $\x_{R_K}$ is indeed feasible with respect to $\Xscr$, we may claim $\lm{\eta}$-stationarity with respect to \eqref{eqn:prob}. 
	\end{remark}

	\begin{theorem}[{\bf Rate and complexity statements for \texttt{VRSQN-ZO}}]\label{thm:sqn}\em
Let $\{\x_k\}$ be generated by Algorithm~\ref{algorithm:quasi-newton} and let
        Assumptions~\ref{ass-2} and \ref{assum:random_vars} hold. Suppose
        $\gamma_k \leq
        \frac{\underline{\lambda}_{\eta,\delta,p}}{\overline{\lambda}_{\eta,\delta,p}^2L_\eta}$
        for all $k$ where $\underline{\lambda}_{\eta,\delta,p}$ and
        $\overline{\lambda}_{\eta,\delta,p}$ are given by
        \eqref{eqn:eigenvalue_bounds}--\eqref{eqn:eigenvalue_bounds2}. Let
        $\delta \eta^2 \leq 4$.  For a given integer $K$, let ${R_K}$ be a random
        variable on $0,\ldots,K-1$ with probability mass function given
        by $\mathbb{P}{[}R_K = j{]} = {\tfrac{ \gamma_j}{\sum_{i=0}^{K-1}   \gamma_i}}$
        for all $0\leq j\leq K-1$.

        \noindent {\bf (i) [Error bound]} If $h_{\eta}^{*}\triangleq {\displaystyle \min_\x} h_\eta(\x)$ and $h_\eta(\x)$ is given by \eqref{eqn:double_smoothed}, then for any $K \ge 1$,
\vspace{-0.1in}
	\begin{align}\label{eqn:thm_sqn_i}
		\mathbb{E}\left[ \|\nabla h_{\eta}(\x_{{R_K}})\|^2\right] \leq \frac{2(h_{\eta}(\fyy{\x_{0}}) - h_{\eta}^{*}) +\left({8} \sqrt{2\pi}L_{\eta}{ n L_0^2}\overline{\lambda}^2_{\eta,\delta,p}\right)\sum_{k=0}^{K-1}\tfrac{\gamma_{k}^2}{N_k}}{\fyy{\fyy{\underline{\lambda}_{\eta,\delta,p}} \sum_{k=0}^{K-1}\gamma_k} }.
	\end{align}
        \noindent {\bf (ii) [Complexity for constant steplength and increasing $N_k$]} {For $0 \le k \le K-1$, let  $N_k = \lceil {n L_0 \eta^{3}}(k+1)^{1+b} \rceil$} and $\gamma_k := \tfrac{\underline{\lambda}_{\eta,\delta,p}}{L_{\eta}\overline{\lambda}^2_{\eta,\delta,p} }$ for any $k \ge 0$, where $a \in \mathbb{R}$ and $b>0 $. Then for any $K\geq 1$, 
\vspace{-0.1in}
		\begin{align*}
		\mathbb{E}\left[ \|\nabla h_{\eta}(\x_{\us{R_K}})\|^2\right] \leq \left(2(h_{\eta}(\fyy{\x_{0}}) - h_{\eta}^{*})L_\eta\left(\tfrac{\overline{\lambda}_{\eta,\delta,p}}{\underline{\lambda}_{\eta,\delta,p}}\right)^2+ \left({{8} n} \sqrt{2\pi}{ L_0^2}L_{\eta} \overline{\lambda}^2_{\eta,\delta,p}\right)\right)\tfrac{1}{K}.
	\end{align*}
If $\delta ={\eta^{-2}L_0^{2} n}$, {$b\geq 0$, and} $ \us{\epsilon}  >0$, to achieve {$\mathbb{E}\left[ \|\nabla h_{\eta}(\x_{\us{R_K}})\|\right]  \leq \epsilon$}, the iteration and sample complexities are $\mathcal{O}\left( {L_0^{4}}{n^{2}}{\eta^{-4}}\epsilon^{-2}\right)$ and 
 ${\mathcal{O}\left({L_0^{9+{8b}} n^{5+{4b}} \eta^{-5-2b}} \epsilon^{-2(2+b)}\right)}$, respectively.
	
        \noindent {\bf (iii) [Complexity for diminishing step and $N_k = 1$]} Let $N_k = 1$ and $\gamma := \tfrac{1}{ {\overline{\lambda}_{\eta,\delta,p}}{\left(16n  L_{\eta} K\right)^{1/2} \pi^{1/4} L_0}}$ for all $0\leq k\leq K-1$. Furthermore, let  $\delta :={\eta^{-2}L_0^{2} n}$.  Then, the following error bound holds for $K \geq \left(\frac{\overline{\lambda}_{\eta,\delta,p}}{\underline{\lambda}_{\eta,\delta,p}}\right)^2\frac{L_\eta}{{16 n} \sqrt{2\pi}{ L_0^2}}$.  
		\begin{align*}
		\mathbb{E}\left[ \|\nabla h_{\eta}(\x_{\us{R_K}})\|^2\right] \leq \left(\left(2(h_{\eta}(\fyy{\x_{0}}) - h_{\eta}^{*})+1\right){\left(16 n L_{\eta} K\right)^{1/2} (2\pi)^{1/4} L_0}\right)\left(\tfrac{\overline{\lambda}_{\eta,\delta,p}}{\underline{\lambda}_{\eta,\delta,p}}\right)\tfrac{1}{\sqrt{K}}.
	\end{align*}
	Further, the iteration and sample complexities are both $\mathcal{O}({n^3}L_0^{5} \eta^{-7} {\epsilon^{-4}})$.
			
%
\end{theorem}
\begin{proof}
	\noindent (i) From Proposition~\ref{prop:sqn_as}, we have  
	$$ \mathbb{E}\left[ h_{\eta}(\x_{k+1})  \, \mid \, \mathcal{F}_k\right] \leq h_{\eta}(\x_k)  - \tfrac{\underline{\lambda}_{\eta,\delta,p}\gamma_k \|{\nabla h_{\eta}}(\x_k) \|^2}{2}   +  \tfrac{\fyy{{{8} n}\sqrt{2\pi}{ L_0^2}L_{\eta}\overline{\lambda}^2_{\eta,\delta,p}}\gamma_{k}^2}{N_k}   .$$
Taking unconditional expectations, summing both sides over $k=0,\ldots,K-1$, 
    dividing the both sides by $\sum_{k=0}^{K-1}\gamma_k$ and invoking the definition of $\mathbb{P}{[}R_{\us{K}} = j{]} = \tfrac{ \gamma_j}{\sum_{i=0}^{K-1}   \gamma_i}$, 
		\begin{align*}
 \tfrac{\underline{\lambda}_{\eta,\delta,p}}{2}\sum_{k=0}^{K-1}\gamma_k \mathbb{E}\left[\|\nabla h_{\eta}(\x_k) \|^2\right]  &  \leq    \mathbb{E}\left[ h_{\eta}(\x_0) \right]  -\mathbb{E}\left[ h_{\eta}(\x_{K})  \right]   + {\fyy{{{8} n} \sqrt{2\pi}{L_0^2}L_{\eta}\overline{\lambda}^2_{\eta,\delta,p}} }\sum_{k=0}^{K-1}\tfrac{\gamma_k^2}{N_k}  \\
	\tfrac{\underline{\lambda}_{\eta,\delta,p}}{2}  \mathbb{E}_{{R_K}}\left[\|\nabla h_{\eta}(\x_k) \|^2\right]  & \leq   \tfrac{\mathbb{E}\left[ h_{\eta}(\x_0) \right]  -\mathbb{E}\left[ h_{\eta}(\x_{K})  \right]   +  {\fyy{{{8} n} \sqrt{2\pi}{ L_0^2}L_{\eta} \overline{\lambda}^2_{\eta,\delta,p}} }\sum_{k=0}^{K-1}\tfrac{\gamma_k^2}{N_k} }{\sum_{k=0}^{K-1}\gamma_k} .
\end{align*}
	Dividing the both sides by $\tfrac{\underline{\lambda}_{\eta,\delta,p}}{2}$ and using $\mathbb{E}\left[ h_{\eta}(\x_{K})  \right] \geq h_{\eta}^{*}$, we obtain the inequality.  
	
\noindent (ii) Substituting {$N_k :=   \lceil a n L_0{\eta^{3}}(k+1)^{1+b} \rceil$} and $\gamma_k = \tfrac{\underline{\lambda}_{\eta,\delta,p}}{L_{\eta}\overline{\lambda}^2_{\eta,\delta,p} }$ in \eqref{eqn:thm_sqn_i}, we obtain 
	\begin{align*}
		\mathbb{E}\left[ \|\nabla h_{\eta}(\x_{\us{R_K}})\|^2\right] \leq \tfrac{2(h_{\eta}(\fyy{\x_{0}}) - h_{\eta}^{*})L_\eta\overline{\lambda}^2_{\eta,\delta,p} }{\underline{\lambda}^2_{\eta,\delta,p}K  }+ \tfrac{{{8}n} \sqrt{2\pi}{ L_0^2} L_{\eta}}{K}  \sum_{k=0}^{K-1} \tfrac{{n^{-1}L_0^{-1}}}{{\eta^{3}}(k+1)^{1+b}}.
	\end{align*}
Note that for any $K\geq 1$, we have {$\sum_{k=0}^{K-1}\frac{1}{(k+1)^{1+b}} \leq 1+b^{-1}$} (see \cite[Lemma 9]{yousefian2017smoothing}). This implies that  the inequality in (ii) holds. To show the complexity results, we proceed as follows. From \eqref{eqn:eigenvalue_bounds}--\eqref{eqn:eigenvalue_bounds2}, 
$\tfrac{\overline{\lambda}_{\eta,\delta,p}}{\underline{\lambda}_{\eta,\delta,p}} =  \left(4p + 1 \right)\left( 1 + \tfrac{\fyy{16 L_\eta}\sqrt{2+\delta}}{\delta} \right)^{2p} 32\delta^{-1}(2+\delta)(p+1)L_\eta^2 .$
Substituting $\delta:= {\eta^{-2}L_0^{2} n}$ and noting that $L_{\eta} = \tfrac{L_{0}{n^{1/2}}+1}{\eta}$, we have 
\begin{align*}
\tfrac{\overline{\lambda}_{\eta,\delta,p}}{\underline{\lambda}_{\eta,\delta,p}} 
&=  \left(4p + 1 \right)\left( 1 + \tfrac{16 (L_{0}{n^{1/2}}+1)\sqrt{2+\eta^{-2}{L_0^2 n}}}{\eta^{-2}{L_{0}^{2}n}} \right)^{2p} \\
& \times 64({\eta^{2}L_0^{-2} n^{-1}})(2+{\eta^{-2}L_0^{2} n})(p+1)\left(\tfrac{L_{0}{\sqrt{n}}+1}{\eta}\right)^2.
\end{align*}
This implies that $\tfrac{\overline{\lambda}_{\eta,\delta,p}}{\underline{\lambda}_{\eta,\delta,p}}= \mathcal{O}({\eta^{-2}L_0^{2} n})$. 
We conclude that the iteration complexity to ensure that {$\mathbb{E}\left[ \|\nabla h_{\eta}(\x_{\us{R_K}})\|\right]  \leq \epsilon$} is $K_\epsilon\triangleq \mathcal{O}\left({L_0^{4}}{n^{2}}{\eta^{-4}}\epsilon^{-2}\right)$, while the sample complexity  is 
\begin{align*}
	\textstyle\sum_{k=0}^{K_\epsilon} N_k & = \textstyle\sum_{k=0}^{K_\epsilon} \lceil {nL_0}{\eta^3} (k+1)^{1+b} \rceil 
	= \mathcal{O}({n L_0 {\eta^{3}}} K_\epsilon^{2+b})\\ 
	&= \mathcal{O}\left({L_0^{9+{8b}} n^{5+{4b}} \eta^{-5-2b}} \epsilon^{-2(2+b)}\right).
\end{align*}

\noindent (iii) First, we note that the lower bound on $K$ is obtained by requiring $\gamma_k \leq \frac{\underline{\lambda}_{\eta,\delta,p}}{\overline{\lambda}_{\eta,\delta,p}^2L_\eta}$ for $\gamma := \tfrac{1}{ {\overline{\lambda}_{\eta,\delta,p}}{\left({16}n  L_{\eta} K\right)^{1/2} \pi^{1/4} L_0}}$. Consider \eqref{eqn:thm_sqn_i} and for a constant stepsize $\gamma$ and $N_k:=1$ for every $k \ge 0$, we have 
		$\mathbb{E}\left[ \|\nabla h_{\eta}(\x_{\us{R_K}})\|^2\right] \leq 
		\tfrac{2(h_{\eta}(\fyy{\x_{0}}) - h_{\eta}^{*}) }{\underline{\lambda}_{\eta,\delta,p} \gamma K } +
		\tfrac{ \left({16} \sqrt{\pi}L_{\eta}{ n L_0^2}\overline{\lambda}^2_{\eta,\delta,p}\right)\gamma }{  \underline{\lambda}_{\eta,\delta,p} }.$
    Substituting $\gamma := \tfrac{1}{{\overline{\lambda}_{\eta,\delta,p}}{\left({16}n  L_{\eta} K\right)^{1/2} \pi^{1/4} L_0}}$, we obtain the bound in (iii). The iteration complexity ({equal to sample complexity when $N_k = 1$ for every $k$}) is obtained from this bound and by noting that {$L_{\eta} = \mathcal{O}({n^{1/2}L_0}\eta^{-1})$ and} from (ii), $\tfrac{\overline{\lambda}_{\eta,\delta,p}}{\underline{\lambda}_{\eta,\delta,p}}= \mathcal{O}({\eta^{-2}L_0^{2} n})$.
\end{proof}
\begin{remark}

\noindent (i) Notably, in the complexity bounds in
    Theorem~\ref{thm:sqn}, the exponents of {$\eta$, $L_0$, and $n$} is invariant {with} the memory
    parameter $p$. This was shown by choosing $\delta:=  {\eta^{-2}L_0^{2} n}$ and
    observing that in the bound on the ratio
    $\tfrac{\overline{\lambda}_{\eta,\delta,p}}{\underline{\lambda}_{\eta,\delta,p}}$,
    the exponents of {$\eta$, $L_0$, and $n$ } are invariant {with} $p$. 
\noindent (ii) Unlike
    \texttt{{VRG-ZO}}, iterates generated by \texttt{VRSQN-ZO} are not
    guaranteed to be feasible. {This implies that a the stationary point obtained by \texttt{VRSQN-ZO} is not guaranteed to be $\eta-$Clarke stationary with respect to the original problem.}
The \texttt{{VRG-ZO}} scheme projects the
    iterate back onto $\Xscr$, enforcing feasibility at each
    iteration, while \texttt{VRSQN-ZO} utilizes a Moreau
    smoothing of the indicator function, merely {penalizing} the infeasibility.
\noindent (iii) {Finally, key benefits of \texttt{VRSQN-ZO} emerge in the form of a ``self-scaling'' behavior core to QN schemes as well as its ability to better deal with ill-conditioning by using larger memory sizes $p$. These benefits are less clear by merely examining the complexity bounds.} 
\end{remark}

\section{Numerical Results}
\label{sec:numerics}
In this section, we examine the performance of \cref{algorithm:zo_nonconvex}
and \cref{algorithm:quasi-newton} {on a test problem from logistic regression}. 

\subsection{Logistic Regression} \label{sec:logistic-regression}
{Given a dataset $\{(\bz_i, y_i)\}_{i=1}^{S}$ where $\bz_i \in \mathbb{R}^{n-1}$ and $y_i \in \{0, 1\}$, we define the logistic regression model as $ {\mathbb{P}[}y_i = 1 | \bz_i{]} := \sigma(\mathbf{w}^{{\top}} \bz_i + {w_0}) := \frac{1}{1 + e^{-(\mathbf{w}^{{\top}} \bz_i + {w_0})}}$.  
 In this setting, the empirical risk (negative log-likelihood) is given by
\vspace{-0.2in}
	\begin{align*}
		\text{Risk}(\mathbf{w}, {w_0}) &= -\tfrac{1}{S} \sum_{i=1}^{S} \left[ y_i \log(\hat{y}_i) + (1 - y_i) \log(1 - \hat{y}_i)\right] \\
		& = -\tfrac{1}{S} \sum_{i=1}^{S} \left[ y_i \log(\sigma(\mathbf{w}^{{\top}} \bz_i + {w_0})) + (1 - y_i) \log(1 - \sigma(\mathbf{w}^{{\top}} \bz_i + {w_0})) \right]
	\end{align*}
	where $\hat{y}_i = \sigma(\mathbf{w}^{{\top}} \bz_i + {w_0})$. We are interested in the case where the $\ell_1$ norm of the weight vector is penalized. That is, we consider the following loss function:				
		$\text{Loss}(\mathbf{w}, {w_0}) = \text{Risk}(\mathbf{w}, {w_0}) + \lambda \| \mathbf{w} \|_1,$
	where $\lambda > 0$ is the regularization parameter, and $\| \mathbf{w} \|_1 = \sum_{j=1}^{n-1} |w_j|$. The optimization problem of interest is therefore to find the weights $\mathbf{w}$ and bias ${w_0}$ that minimize the regularized loss function:
	\begin{align*}
		\min_{(\mathbf{w}, {w_0}) \in \Xscr } \left\{ -\tfrac{1}{S} \sum_{i=1}^{S} \left[ y_i \log(\hat{y}_i) + (1 - y_i) \log(1 - \hat{y}_i) \right] + \lambda \sum_{j=1}^{n-1} |w_j| \right\} .
	\end{align*} In what follows, we denote the concatenation of the weight vector $\mathbf{w}$ and the bias ${w_0}$ by $\x$, and define $f(\x) := \text{Loss}(\mathbf{w}, {w_0})$. In this setting, a single sample $\xi$ corresponds to one $\left(\bz ,y\right)$ pair.}
{
	We consider this problem for a number of dimensions in \far{Table}~\ref{tbl:regression_results}, assuming that the number of important features is $\lceil 0.2 n\rceil$, with the remaining features containing {multivariate Gaussian noise (with mean zero and  unit variance)}.
	{We denote the number of iterations {during which} ``damping'' is active by $k_{\text{damp}}$.}} {In Tables~\ref{tab:consolidated-vrsqn-zo}--\ref{tab:consolidated-vrg-zo}, we present additional results for   \texttt{VRSQN-ZO} and \texttt{VRG-ZO}, respectively, under different stepsize sequences and sampling rates. In these tables, the sampling budget is set to
$5e4$ and $N_k := \lceil 2 + a k \rceil $. We
observe that smaller values of $a$ lead to lower
sampling levels at a given $k$, more overall steps,
and slightly better accuracy at the cost of slightly
larger CPU times. We also observe that for
\texttt{VR-SQN}, the empirical infeasibility bounds
correspond with the choice of $\eta$, as proven in the
theoretical claims.}} 
	\begin{table}[htb]
		\centering
		\caption{\footnotesize Logistic Regression via \texttt{VRSQN-ZO} and \texttt{VRG-ZO}. We consider settings where $\Xscr := \mathbb{R}^n$ in order to validate
		 that reasonable solutions are obtained. In this setting, ``accuracy'' refers to the total number of correct predictions divided by the total number of predictions. ``Precision'' 
		 refers to the portion of ``positives'' identified by the model that were indeed true positives. ``Recall'' refers to the fraction of true positives in the dataset that were
		  identified by the model. The purpose of these metrics is merely to verify that these schemes are indeed ``solving'' a logistic regression problem.}
		\label{tbl:regression_results}
		\scriptsize
		\begin{tabular}{ccccccccc}
			\toprule
			$n$ & \multicolumn{2}{c}{Accuracy} 
			& \multicolumn{2}{c}{Precision} 
			& \multicolumn{2}{c}{Recall} 
			& \multicolumn{2}{c}{$\nabla f(\x_{R})$} \\
			\cmidrule(lr){2-3} \cmidrule(lr){4-5} \cmidrule(lr){6-7} \cmidrule(lr){8-9}
			& \texttt{VRSQN-ZO} & \texttt{VRG-ZO} & \texttt{VRSQN-ZO} & \texttt{VRG-ZO} & \texttt{VRSQN-ZO} & \texttt{VRG-ZO} & \texttt{VRSQN-ZO} & \texttt{VRG-ZO} \\
			\midrule
			5   & 9.6e-1 & 9.6e-1 & 9.2e-1 & 9.3e-1 & 9.9e-1 & 9.9e-1 & 2.8e-1 & 1.1e-1 \\
			10  & 9.9e-1 & 9.9e-1 & 9.8e-1 & 9.8e-1 & 9.9e-1 & 9.9e-1 & 2.2e-1 & 8.0e-2 \\
			50  & 9.3e-1 & 9.4e-1 & 9.4e-1 & 9.4e-1 & 9.4e-1 & 9.5e-1 & 1.7e-1 & 1.3e-1 \\
			100 & 9.4e-1 & 9.4e-1 & 9.3e-1 & 9.4e-1 & 9.4e-1 & 9.5e-1 & 1.8e-1 & 2.1e-1 \\
			\bottomrule
		\end{tabular}
	\end{table}
	\begin{table}[ht]
		\scriptsize
		\centering
				\caption{{\scriptsize \texttt{VRSQN-ZO} under $\{\gamma_k\}$ sequences and sampling rates.
				 $\mathbb{E}\left[\left\|\Pi_{\Xscr}\left[\x_K \right] - \x_K\right\|\right]$ is denoted by $\textbf{infeas}_{K}$, $\mathbb{E}\left[\|\nabla f(\x_{R})\| \right]$
				  by $G_{R}^{\texttt{SQN}}$, $\mathbb{E}\left[\|\nabla f(\x_{K})\| \right]$ by $G_{K}^{\texttt{SQN}}$, and $\mathbb{E}\left[k_{\text{damp}}\right]$
				   by $\hat{\textbf{k}}_{\textbf{damp}}$.}}
		
		\begin{tabular}{>{\centering\arraybackslash}p{1.85cm}ccccccc}
			\toprule
			\textbf{$\gamma_k$} & $(a)$ & $G_{R}^{\texttt{SQN}}$ & $G_{K}^{\texttt{SQN}}$ & $\mathbb{E}\left[ f(\x_{K}) \right]$ & CPU (s) & $\textbf{infeas}_{K}$ & $\hat{\textbf{k}}_{\textbf{damp}} / K$ \\ 
			\midrule
			\multirow{3}{*}{$\frac{1}{100}$}                           
			& $0.01$ & 0.18  & 0.18 & 0.24 & 6.75  & 0.36 & 211.8 / 2659\\
			& $0.1$  & 0.20  & 0.18 & 0.24 & 5.32  & 0.35 & 35.2 / 946 \\
			& $1.0$  & 0.19  & 0.18 & 0.24 & 4.80  & 0.35 & 16.6 / 310 \\
			\midrule
			\multirow{3}{*}{$\left(1 +\tfrac{\sqrt{k + 1}}{100}\right)^{-1}$} 
			& $0.01$ & 0.18 & 0.18 & 0.24 & 5.74 & 0.36 & 183.0 / 2659  \\
			& $0.1$  & 0.20 & 0.18 & 0.24 & 4.54 & 0.35 & 34.2 / 946 \\
			& $1.0$  & 0.22 & 0.18 & 0.24 & 6.34 & 0.36 & 16.4 / 310 \\
			\midrule
			\multirow{3}{*}{$\left(1 +\tfrac{k + 1}{100}\right)^{-1}$}      
			& $0.01$ & 0.18 & 0.18 & 0.24 & 7.46 & 0.36 & 62.6 / 2659 \\
			& $0.1$  & 0.18 & 0.18 & 0.24 & 5.31 & 0.36 & 16.6 / 946  \\
			& $1.0$  & 0.19 & 0.18 & 0.24 & 3.33 & 0.35 & 12.4 / 310  \\
			\bottomrule
		\end{tabular}
		\label{tab:consolidated-vrsqn-zo}
		{\scriptsize ($B = 5.0e4$, $n = 5$, $\eta = 0.1$)}
	\end{table}
	\begin{table}[ht]
		\scriptsize
		\centering
		\caption{{\footnotesize \texttt{VRG-ZO} under different {stepsize} sequences and sampling rates. Here we denote $G_{R}:=\mathbb{E}\left[\left\| \x_{R} - \Pi_{\Xscr}\left[x_R - \nabla f(\x_{R})\right]\right\|^2\right]$, and $G_{K}:=\mathbb{E}\left[\left\| \x_{K} - \Pi_{\Xscr}\left[x_K - \nabla f(\x_{K})\right]\right\|^2\right]$.} }
		\label{tab:consolidated-vrg-zo}
		\begin{tabular}{cccccc}
			\toprule
			$\gamma_k$ & $(a)$ & $G_{R}$ & $G_{K}$ & $\mathbb{E}\left[f(\x_K)\right]$ & CPU Time (s) \\ 
			\midrule
			\multirow{3}{*}{$\frac{1}{100}$}                           
			& $0.01$ & 0.24 & 0.24 & 0.32 & 3.3 \\
			& $0.1$  & 0.28 & 0.24 & 0.32 & 2.5 \\
			& $1.0$  & 0.34 & 0.24 & 0.32 & 2.5 \\
			\midrule
			\multirow{3}{*}{$\left(1 +\tfrac{\sqrt{k + 1}}{100}\right)^{-1}$} 
			& $0.01$ & 0.20 & 0.19 & 0.35 & 8.4 \\
			& $0.1$  & 0.30 & 0.19 & 0.35 & 3.7 \\
			& $1.0$  & 0.36 & 0.21 & 0.38 & 2.3 \\
			\midrule
			\multirow{3}{*}{$\left(1 +\tfrac{k + 1}{100}\right)^{-1}$}      
			& $0.01$ &  0.24 & 0.24 & 0.32 & 3.0 \\
			& $0.1$  & 0.28 & 0.24 & 0.32 & 2.5  \\
			& $1.0$  & 0.35 & 0.25 & 0.33 & 3.1 \\
			\bottomrule
		\end{tabular}
	\end{table}
	\subsection{Minimum of Two Noise-afflicted Quadratics}
	\label{sec:min-quad}
	Let ${\bxi}$ be a uniform random variable on $[0,2].$ Define $f:\mathbb{R}^d \mapsto \mathbb{R}$ as $f(\x) = \mathbb{E} \left[  \min \left[ f_1(\x,{\bxi}),f_2(\x, {\bxi}) \right] \right]$,
	where $f_1(\x, {\bxi}) = \sum_{i}^{d} (x_i - \bxi )^2$ and  $f_2(\x, {\bxi}) = \sum_{i}^{d} (x_{i} + \bxi )^2 .$ 
	 The feasible region, $\Xscr$, is a $n$-dimensional cube of width $10$. {In Tables~\ref{tab:consolidated-vrsqn-zo-min-quadratic} and \ref{tab:consolidated-vrg-zo-min-quadratic} } we compare the performance of \texttt{VRG-ZO} and \texttt{VRSQN-ZO} equipped with the same sample budget of 5.0$e$6, {varying batch size sequences} after 20 replications. For \texttt{VRG-ZO}, $\ell$ was set to $\lceil K/2 \rceil$.
	 \texttt{VRSQN-ZO} appears to outperform \texttt{VRG-ZO} in this setting. 	 
	\begin{figure}[ht]
		\centering
		\begin{subfigure}[b]{0.4\textwidth}
			\includegraphics[width=\textwidth]{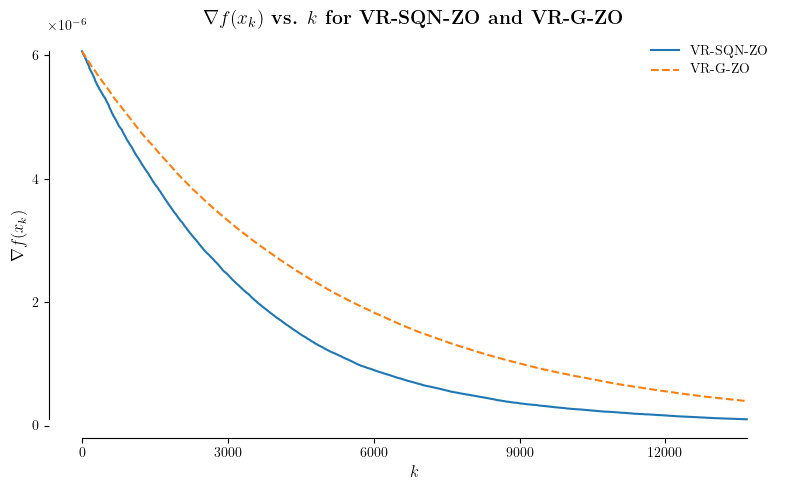}
		\end{subfigure}
		\hfill
		\begin{subfigure}[b]{0.4\textwidth}
			\includegraphics[width=\textwidth]{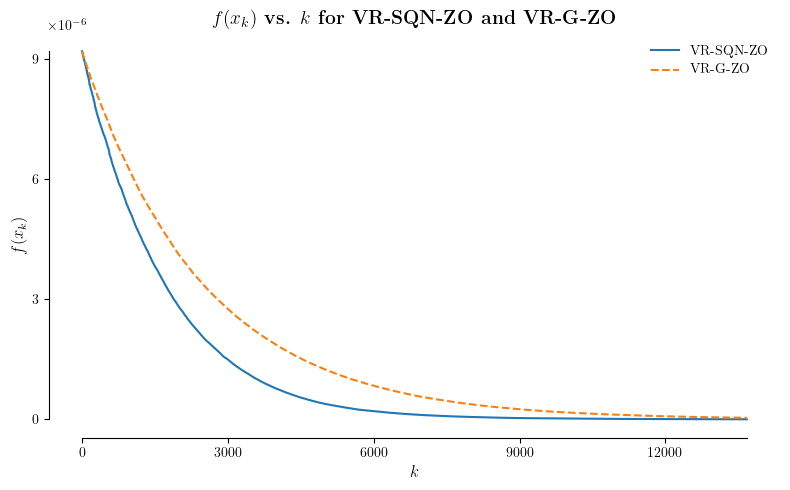}
		\end{subfigure}
		\caption{\texttt{VRSQN-ZO} vs. \texttt{VRG-ZO} on the minimum of two noise-afflicted quadratics.}
		\label{fig:vrsqn_vs_vrgzo}
\vspace{-0.2in}
	\end{figure}
	\begin{table}[ht]
		\scriptsize
		\centering
				\caption{{\scriptsize \texttt{VRSQN-ZO} under $\{\gamma_k\}$ sequences and sampling rates.
				 $\mathbb{E}\left[\left\|\Pi_{\Xscr}\left[\x_K \right] - \x_K\right\|\right]$ is denoted by $\textbf{infeas}_{K}$, $\mathbb{E}\left[\|\nabla f(\x_{R})\| \right]$
				  by $G_{R}^{\texttt{SQN}}$, $\mathbb{E}\left[\|\nabla f(\x_{K})\| \right]$ by $G_{K}^{\texttt{SQN}}$, and $\mathbb{E}\left[k_{\text{damp}}\right]$
				   by $\hat{\textbf{k}}_{\textbf{damp}}$.}}
		
		\begin{tabular}{>{\centering\arraybackslash}p{1.85cm}ccccccc}
			\toprule
			\textbf{$\gamma_k$} & $(a)$ & $G_{R}^{\texttt{SQN}}$ & $G_{K}^{\texttt{SQN}}$ & $\mathbb{E}\left[ f(\x_{K}) \right]$ & CPU (s) & $\textbf{infeas}_{K}$ & $\hat{\textbf{k}}_{\textbf{damp}} / K$ \\ 
			\midrule
			\multirow{3}{*}{$\frac{1}{100}$}                           
			& $0.01$ & 5.9e-7 & 2.9e-7 & 2.0e-8 & 21.26 & 0.0 / 9465 \\
			& $0.1$  & 2.93e-6 & 2.07e-6 & 1.08e-6 & 10.30 & 2.4 / 3107 \\
			& $1.0$  & 4.20e-6 & 3.81e-6 & 3.68e-6 & 4.83 & 2.4 / 994   \\
			\midrule
			\multirow{3}{*}{$\left(1 +\tfrac{\sqrt{k + 1}}{100}\right)^{-1}$} 
			& $0.01$ & 3.42e-6 & 3.02e-6 & 2.32e-6 & 11.47 & 0.0 / 9465  \\
			& $0.1$  & 4.94e-6 & 4.65e-6 & 5.44e-6 & 3.56 & 2.4 / 3107   \\
			& $1.0$  & 4.87e-6 & 4.70e-6 & 5.60e-6 & 1.58 & 3.2 / 994    \\
			\midrule
			\multirow{3}{*}{$\left(1 +\tfrac{k + 1}{100}\right)^{-1}$}      
			& $0.01$ & 4.69e-6 & 4.64e-6 & 5.47e-6 & 9.22 & 0.0 / 9465  \\
			& $0.1$  & 5.86e-6 & 5.76e-6 & 8.35e-6 & 3.10 & 4.0 / 3107  \\
			& $1.0$  & 5.21e-6 & 5.11e-6 & 6.63e-6 & 1.33 & 2.6 / 994  \\
			\bottomrule
		\end{tabular}
		\label{tab:consolidated-vrsqn-zo-min-quadratic}
		{\scriptsize ($B = 5.0e4$, $n = 12$, $\eta = 0.1$)}
	\end{table}
	\begin{table}[ht]
		\scriptsize
		\centering
		\caption{{\scriptsize \texttt{VRG-ZO} under different {stepsize} sequences and sampling rates. Here we denote $G_{R}:=\mathbb{E}\left[\left\| \x_{R} - \Pi_{\Xscr}\left[x_R - \nabla f(\x_{R})\right]\right\|^2\right]$, and $G_{K}:=\mathbb{E}\left[\left\| \x_{K} - \Pi_{\Xscr}\left[x_K - \nabla f(\x_{K})\right]\right\|^2\right]$.} }
		\label{tab:consolidated-vrg-zo-min-quadratic}
		\begin{tabular}{cccccc}
			\toprule
			$\gamma_k$ & $(a)$ & $G_{R}$ & $G_{K}$ & $\mathbb{E}\left[f(\x_K)\right]$ & CPU Time (s) \\ 
			\midrule
			\multirow{3}{*}{$\frac{1}{100}$}                           
			& $0.01$ & 1.29e-6 & 8.1e-7 & 1.70e-7 & 15.9 \\
			& $0.1$  & 4.17e-6 & 3.49e-6 & 3.05e-6 & 8.20  \\
			& $1.0$  & 4.82e-6 & 4.59e-6 & 5.34e-6 & 3.40  \\
			\midrule
			\multirow{3}{*}{$\left(1 +\tfrac{\sqrt{k + 1}}{100}\right)^{-1}$} 
			& $0.01$ & 4.03e-6 & 3.76e-6 & 3.59e-6 & 4.51  \\
			& $0.1$  & 5.55e-6 & 5.38e-6 & 7.27e-6 & 1.71  \\
			& $1.0$  & 5.20e-6 & 5.10e-6 & 6.60e-6 & 0.86  \\
			\midrule
			\multirow{3}{*}{$\left(1 +\tfrac{k + 1}{100}\right)^{-1}$}      
			& $0.01$ & 4.90e-6 & 4.74e-6 & 5.71e-6 & 0.83  \\
			& $0.1$  & 4.73e-6 & 4.11e-6 & 4.23e-6 & 1.74  \\
			& $1.0$  & 2.16e-6 & 1.65e-6 & 7.0e-7 & 4.20  \\
			\bottomrule
		\end{tabular}
\vspace{-0.2in}
	\end{table}
\section{Concluding {R}emarks}\label{sec:conc}
While a significant amount of prior research has analyzed nonsmooth and
nonconvex optimization problems, much of this effort has relied on either the
imposition of structural assumptions on the problem or required weak convexity,
rather than general nonconvexity. Little research, if any, is available in
stochastic regimes to contend with general nonconvex and nonsmooth optimization
problems. To this end, we develop a randomized smoothing framework which
allows for claiming that a stationary point of the $\eta$-smoothed problem is an
$\eta$-{Clarke} stationary point for the original problem. Via
a suitable residual function that provides a metric for stationarity
for the smoothed problem, we present a zeroth-order framework reliant on
utilizing sampled function evaluations. In this setting, {we show that} the residual function of the smoothed
problem tends to zero almost surely along the generated sequence. To
compute an $\x$ that ensures that the expected norm of the residual of the $\eta$-smoothed problem is within
$\epsilon$, we proceed to show that no more than \vrgzoIterationComplexity ~projection steps and \vrgzoSampleComplexity ~function evaluations are required. Further, we propose a zeroth-order stochastic
quasi-Newton scheme by combining randomized and Moreau
smoothing. We establish an almost-sure convergence result and derive the
corresponding iteration and sample complexities of
\vrsqnzoIterationComplexity ~and  
\vrsqnzoSampleComplexity, respectively. These
results appear to be novel in addressing constrained nonsmooth
nonconvex stochastic optimization.

\bibliographystyle{plain}
\bibliography{demobib_v1,wsc11-v03a,ref_paIRIG_v01_fy}


\end{document}